\documentclass{article}
\usepackage{amsthm,amsmath,amssymb,amscd,bm,graphicx,color,hyperref}

\usepackage{mycommands}

\begin{document}

\title{Floer cohomology of immersed Lagrangian spheres in smoothings of $A_N$ surfaces}
\author{Garrett Alston\footnote{The Institute of Mathematical Sciences and Department of Mathematics, The Chinese University of Hong Kong; Email: \url{galston@math.cuhk.edu.hk}\newline Partially supported by grant NSF-DMS 1108397.}}
\date{}


\maketitle
\begin{abstract}
  We calculate the self-Floer cohomology with $\zz_2$ coefficients of some immersed Lagrangian spheres in the affine symplectic manifolds
  \begin{displaymath}
    \sett{\za\zb-(\zc-1)(\zc-2)\cdots(\zc-N)=0}\subset\cc^3.
  \end{displaymath}
  The immersed spheres are exact and graded.
  Moreover, they satisfy a positivity assumption that allows us to apply the methods of \cite{alston-fciegl} to calculate the Floer cohomology as follows.
  Given auxiliary data a Morse function on $S^2$ and a time-dependent almost complex structure, the Floer cochain complex is the Morse complex plus two generators for each self-intersection point of the Lagrangian sphere.
  The Floer differential is defined by counting combinations of Morse flow lines and holomorphic strips.
  Using a Lefschetz fibration allows us to explicitly calculate all holomorphic strips and describe the Floer differential.
  For most of the immersed spheres the Floer differential is zero (with $\zz_2$-coefficients).
\end{abstract}

\section{Introduction}
\label{sec:introduction}
Exact Lagrangian submanifolds of exact symplectic manifolds are an important class of objects because they often say interesting things about the symplectic manifold.
For example, the trivial symplectic manifold $(\rr^{2n},\omega_0)$ does not have any exact Lagrangian submanifolds.
As another example, the existence of an exact Lagrangian submanifold in a Liouville manifold $M$ implies the non-vanishing of the symplectic cohomology of $M$.
Symplectic cohomology in turn has many applications, such as to the existence of Reeb orbits on the boundary of $M$.

One way to generalize the study of exact Lagrangian submanifolds is to consider \textit{immersed} exact Lagrangians.
These objects should also have interesting applications to symplectic geometry.
In this paper we study the Floer cohomology of a family of immersed Lagrangian spheres in the smoothing of an $A_n$ surface.

We now introduce the objects we will study and state the main theorems of the paper.
Let
\begin{displaymath}
  \poly = \za\zb-(\zc-1)(\zc-2)\cdots(\zc-N)
\end{displaymath}
and
\begin{displaymath}
  \mnfld=\sett{\poly=0}\subset\cc^3.
\end{displaymath}
We equip $\mnfld$ with the standard symplectic form $\sympl$, standard complex structure $\acs$, standard one form $\oneform$ with $d\oneform=\omega$, and standard holomorphic volume form $\holvf = \mathrm{Res}(d \za \wedge d\zb \wedge d\zc/\poly)$.
All of this structure makes $\mnfld$ an exact graded K\"ahler manifold.

The functions $\hamone=\frac{1}{2}|\za|^2-\frac{1}{2}|\zb|^2$, $\hamtwo=|\zc|^2$ define a singular Lagrangian torus fibration on $\mnfld$.
The fibers which contain focus-focus singularities are immersed Lagrangian spheres.
These are precisely the fibers where $\hamone=0$ and $\hamtwo\in\sett{1,\ldots,N}$; we denote these fibers by $\lagstd$.
The main result of this paper is a calculation of the Floer cohomology of these immersed spheres.
The result is
\begin{theorem}[Theorem \ref{thm:1}]
  \label{thm:3}
  The self-Floer cohomology (with $\zz_2$-coefficients) of $\lagstd$ is $0$ if $r=1$, and isomorphic to $(\zz_2)^4$ if $r>1$.
  More precisely, in the latter case the Floer cohomology has dimension one in degrees $-1,0,2,3$ and is $0$ elsewhere.
\end{theorem}

More generally, we can calculate the Floer cohomology of many other immersed Lagrangian spheres.
To describe these Lagrangians, note that the map
\begin{displaymath}
  \leffib:\mnfld\to\cc,\quad (\za,\zb,\zc)\mapsto\zc
\end{displaymath}
is a Lefschetz fibration.
If $\gamma$ is an embedded loop in the base that passes through exactly one critical value, then the union of vanishing cycles over $\gamma$ forms an immersed sphere which we denote as $\Sigma_\gamma$.
The result is
\begin{theorem}[Theorem \ref{thm:2}]
  \label{thm:4}
  If the interior of $\gamma$ contains no critical values of $\leffib$, then the Floer cohomology of $\Sigma_\gamma$ is trivial.
  Otherwise, the Floer cohomology of $\Sigma_\gamma$ is isomorphic to $(\zz_2)^4$.
  More precisely, it has dimension $1$ in degrees $-1,0,2,3$.
\end{theorem}

Floer cohomology for immersed Lagrangians was developed by Akaho and Joyce in \cite{MR2785840}.
In their setup, they construct a filtered $A_\infty$-algebra associated to a compact immersed Lagrangian $\iota:L\to M$ with transverse self-intersection points (here $L$ is a smooth manifold, $\iota(L)$ is Lagrangian, and $M$ is a symplectic manifold).
The underlying vector space (over a Novikov ring $\Lambda$) for the $A_\infty$-algebra can be thought of as $H^*(L;\Lambda)\oplus \Lambda R$ where $R$ is a finite set that has two elements for each self-intersection point of $\iota$.
If $b$ is a solution of the Maurer-Cartan equation for the $A_\infty$-algebra then the Floer cohomology of $(\iota,b)$ is defined.

The full $A_\infty$-algebra of a Lagrangian is difficult to calculate, and hence it is difficult to compute Floer cohomology from this definition.
We therefore use a different definition of Floer cohomology in this paper.
We give a brief overview here; see Section \ref{sec:floer-cohom-lagstd} for more details and \cite{alston-fciegl} for complete details.
First, for the underlying chain complex we will take
\begin{displaymath}
  \cm(f;\zz_2)\oplus\zz_2R
\end{displaymath}
where $R$ is as above, $f$ is a Morse function on $L$, and $\cm(f;\zz_2)$ is the Morse complex.
Motivated by \cite{MR2555932} and \cite{MR2863919}, we define the Floer differential by counting pearly trajectories, which are strings of Morse flow lines and holomorphic discs.
In our formulation we use holomorphic strips instead of discs, and allow time dependent almost complex structures.
The $\pm\infty$ ends of the strips converge to either branch points\footnote{I.e.\ self-intersection points of the immersed Lagrangian.} of the immersed Lagrangian or smooth points (non-branch points).
Exactness precludes existence of strips with no branch jumps, hence strips will only appear in pearly trajectories that start or stop on generators coming from $R$.
Thus, our pearly trajectories will consist of one (maybe partial) Morse flow line and at most two discs, each of which has a branch jump.
See Figure \ref{fig:trajectories}.
\begin{figure}[htbp]
  \label{fig:trajectories}
  \centering
  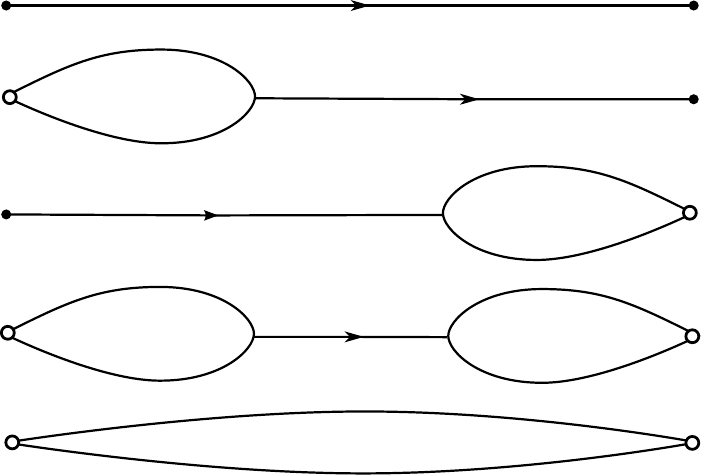
  \caption{The five types of pearly trajectories. Lines are (maybe partial) Morse flow lines. Lines connect to the $\pm\infty$ ends of the strips. Closed dots at the ends of trajectories are critical points. Open circles represent branch jumps. Branch jumps correspond to elements of $R$. All other points on the boundary of the strips are smooth. The positivity assumption implies that the fourth type actually does not contribute for dimension reasons (we show it for completeness).} 
\end{figure}

See Theorem \ref{thm:5} for a precise statement of what is proved in \cite{alston-fciegl} and Section 7.1 for the precise definition of Floer cohomology.
Note that the positivity assumption in Theorem \ref{thm:5} can be thought of as an unobstructedness assumption for the immersed Lagrangian---it is needed to rule out problems arising from disc bubbling.
The definition of Floer cohomology used in this paper agrees with the more standard definition---namely, taking two copies of the Lagrangian, Hamiltonian perturbing one, and then counting strips connecting intersection points.

In Sections \ref{sec:discs-strips}, \ref{sec:analytic-setup} and \ref{sec:moduli-spac-holom} we explain the construction of the moduli spaces of strips and discs.
The strips are used to define the pearly trajectories and the discs are used to help model possible degenerations.
A time dependent complex structure $\depj{t}{t}$ is generic (and hence is a valid choice to define the Floer cohomology) if all the moduli spaces of strips with all possible boundary punctures are regular and tranverse to the stable and unstable manifolds of the Morse function.
In Section \ref{sec:class-holom-curv} we classify all strips with boundaries on the Lagrangians $\lagstd$ with respect to the standard time independent complex structure.
This classification follows relatively easily from Lefschetz fibration considerations; compare for example to the calculation in Example 17.3 of \cite{seidel-fcpclt}.
We then show that all these moduli spaces are regular, hence can be used to calculate Floer cohomology as defined above.
The main lemma that is used to calculate Floer cohomology is Lemma \ref{lemma:1}, which in turn is based on the classification in Section \ref{sec:class-holom-curv}.

We end this introduction with some remarks on possible further work.
First, it would be interesting to find the product structure and full $A_\infty$-structure on the Lagrangians.
We conjecture that the only non-trivial products, other than the products involving the unit, are
\begin{displaymath}
  [(\q,\p)]\cdot[(\p,\q)]=[(\p,\q)]\cdot [(\q,\p)] = [\pmax].
\end{displaymath}
Here, $[(\p,\q)], [(\p,\q)], [\pmax]$ are generators of the Floer cohomology in degrees $-1,3,2$ respectively.
See Section \ref{sec:calc-floer-cohom} for a full explanation of the notation.
The product can likely be defined by counting a combination of holomorphic discs and Morse flow lines, as in \cite{MR2555932} or \cite{MR2863919}.
It seems that the only thing that contributes to the product in this example is a constant disc (with branch jumps\footnote{That is, the map $\ell$ in Definition \ref{dfn:9} takes different values on the components of the boundary of the disc minus the marked points.}) connected to  a Morse flow line emanating from $\pmax$.
However, we have not yet proved that this correctly computes the product.

Second, the immersed Lagrangians can be constructed from embedded Lagrangian spheres using the construction described in Section (16h) of \cite{seidel-fcpclt}.
In the terminology of Lefschetz fibrations (see Section \ref{sec:geom-lefsch-fibr}), the immersed spheres are matching cycles over loops in the base that start and stop at a critical value of the fibration.
Such loops can be obtained as perturbations of concatenations of two paths connecting two critical values of the fibration.
See Figure \ref{fig:curves}.
The matching cycles over the paths are embedded spheres that intersect over the critical values.
Doing appropriate Lagrangian surgery on the spheres at one of the intersection points will result in the immersed sphere.
It would be interesting to explicitly describe the relationship between all of these objects in the twisted Fukaya category.
\begin{figure}[htbp]
  \label{fig:curves}
  \centering
  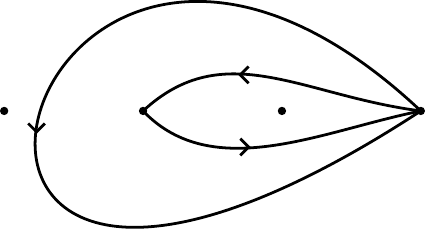
  \caption{Two paths and a loop in the base of the Lefschetz fibration. The dots represent critical values of the fibration.}
\end{figure}

\section{Immersed Lagrangian spheres}
\label{sec:torus-fibration}

In this section we define the Lagrangian spheres $\lagstd$ that we will study, and explain how to see them from the point of view of Lagrangian torus fibrations and Lefschetz fibrations.
Later on when we classify holomorphic discs the Lefschetz fibration point of view will be very useful.

The self-intersection point of $\lagstd$ has a special role to play in immersed Floer theory and we fix some notation regarding it.
We also write down explicit coordinates on $\lagstd$, as well as fix a grading and discuss exactness of $\lagstd$.
Our grading conventions agree with those in \cite{seidel-fcpclt}.

\subsection{Torus fibration perspective}
Consider the functions
\begin{eqnarray}
  \label{eq:2}
\nonumber  \hamone:\mnfld\to\rr,\quad& (\za,\zb,\zc)\mapsto \frac{1}{2}|\za|^2-\frac{1}{2}|\zb|^2,\\
  \hamtwo:\mnfld\to\rr,\quad & (\za,\zb,\zc)\mapsto |\zc|^2.
\end{eqnarray}
\begin{lemma}
  $\hamone$ and $\hamtwo$ Poisson commute, that is
  \begin{displaymath}
    \{\hamone,\hamtwo\}=0.    
  \end{displaymath}
\end{lemma}
\begin{proof}
  It is easy to check that the Hamiltonian vector field\footnote{Our convention is that $\sympl(\hamvf{\hamone},\cdot)=d\hamone(\cdot)$.}
 of $\hamone$ is
  \begin{displaymath}
    \hamvf{\hamone}=-i\za\pd{}{\za}+i\zb\pd{}{\zb}.
  \end{displaymath}
  Thus
  \begin{displaymath}
    \{\hamone,\hamtwo\}=-d\hamtwo(\hamvf{\hamone})=\hamvf{\hamone}(\hamtwo)=0.
  \end{displaymath}
\end{proof}

\begin{corollary}
  The map
  \begin{displaymath}
    (\hamone,\hamtwo):\mnfld \to \base:=\rr\times\rr_{\geq0}
  \end{displaymath}
  is a Lagrangian torus fibration (with singularities).
  The singular locus is the subset
  \begin{displaymath}
    \rr\times\sett{0}\cup\sett{(0,1^2),\ldots,(0,N^2)}\subset B.
  \end{displaymath}
  The fibers over $(0,1^2),\ldots,(0,N^2)$ are focus-focus fibers; hence they are immersed Lagrangian spheres.
  Topologically, they can also be viewed as tori with a cycle collapsed to a point.
\end{corollary}

The fibers that are immersed spheres will be the main subject of study in this paper.
\begin{definition}
  \label{dfn:28}
  For $r\in\sett{1,\ldots,N}$ let
  \begin{displaymath}
    \lagstd = \sett{\hamone=0,\ \hamtwo=r^2}\subset\mnfld
  \end{displaymath}
  denote the Lagrangian torus fiber over $(0,r^2)\in\base$.
  Each $\lagstd$ is an immersed Lagrangian sphere.
\end{definition}

We will need explicit formulas for the immersions of these Lagrangian spheres.
Let $S^2\subset\rr^3$ denote the standard sphere.
Define cylindrical coordinates on $S^2$ by
\begin{equation}
  \label{eq:1}
  (a,e^{ib})\in(-\pi,\pi)\times S^1 \mapsto (\cos (a/2)\cos (b),\cos(a/2)\sin(b),\sin(a/2))\in S^2.
\end{equation}
We also define (two patches of) rectangular coordinates on $S^2$ by
\begin{displaymath}
  (x,y)\mapsto (x,y,\pm\sqrt{1-x^2-y^2}).
\end{displaymath}
The relationship between the coordinates is
\begin{displaymath}
  x=x(a)=\cos(a/2)\cos(b),\quad y=y(a)=\cos(a/2)\sin(b).
\end{displaymath}  
Let $\rho=\rho(a)=x^2+y^2$ and let $f:(-\pi,\pi)\to(-\pi,\pi)$ be a function such that $f'>0$, $f(a)=a$ near $a=0$, $f(a)=\pi-\rho$ near $a=\pi$, and $f(a)=-\pi+\rho$ near $a=-\pi$.
Note that $f$ can be viewed as a smooth function on $S^2$ (it does not depend on $e^{ib}$).
\begin{definition}
  \label{dfn:29}
  For $r\in\sett{1,\ldots,N}$ let $\imm:S^2\to \lagstd$ be the immersion defined in cylindrical coordinates on $S^2$ by the formula
  \begin{displaymath}
    \imm:(a,e^{ib})\mapsto (e^{ib}\xi(a),e^{-ib}\xi(a),-re^{if(a)}),
  \end{displaymath}
  with
  \begin{displaymath}
    \xi(a)=\sqrt{-re^{if(a)}-1}\sqrt{-re^{if(a)}-2}\cdots\sqrt{-re^{if(a)}-N}.
  \end{displaymath}
  Here we define the square root function by setting $\sqrt{-1}=i$ and using analytic continuation.
\end{definition}

\begin{lemma}
  \label{lemma:2}
  $\imm$ is a smooth immersion with one transverse self-intersection.
  The preimage of the self intersection points correspond to $a=\pm\pi$ in spherical coordinates.
\end{lemma}
\begin{proof}
  The only nontrivial thing to check is that the formula for $\imm$ extends smoothly over the two points $a=\pm\pi$ (which are not technically in the cylindrical coordinate chart), and is an embedding at these points.
  Let us concentrate near $a=\pi$; the calculation near $a=-\pi$ is similar.

  Let $\eta(x,y)=\sqrt{e^{if(a)}+1}$.
  Near $(x,y)=(0,0)$,
  \begin{displaymath}
    e^{if(a)}+1=\rho(i+O(\rho)),\quad \eta=\rho^{1/2}\sqrt{i+O(\rho)}.
  \end{displaymath}
  Also, $\cos b =x\rho^{-1/2}$ and $\sin b=y\rho^{-1/2}$, and hence
  \begin{displaymath}
    e^{ib}\eta=(x+iy)\sqrt{i+O(\rho)},\quad e^{-ib}\eta=(x-iy)\sqrt{i+O(\rho)}.
  \end{displaymath}
  Thus $e^{\pm i b}\eta$ is a smooth function of $(x,y)$ near $(0,0)$.
  It is clear that $\xi/\eta$ is smooth, hence
  \begin{displaymath}
    e^{\pm i b}\xi=\frac{\xi}{\eta}\cdot e^{\pm i b}\eta
  \end{displaymath}
  is smooth.
  An easy computation shows that the derivative at $(0,0)$ is invertible, thus proving the lemma.
\end{proof}

The preimages of the self-intersection point will have a special role to play, so we give names to these points.
\begin{definition}
  \label{dfn:3}
  Let $\p\in S^2$ be the point that corresponds to $a=\pi$ and let $\q\in S^2$ be the point that corresponds to $a=-\pi$.
\end{definition}

Near the self-intersection point in the image $\lagstd=\imm(S^2)$, the Lagrangian has two branches; we call these branches the $\p$-branch and the $\q$-branch depending on whether $\p$ or $\q$ is in the lift to $S^2$ of the branch.
\subsection{Lefschetz fibration perspective}
\label{sec:lefsch-fibr-persp}
It is also possible to see the Lagrangian torus fibration from the point of view of Lefschetz fibrations.
In fact, this point of view is better for the purpose of studying holomorphic curves.

The Lefschetz fibration we will use is the map
\begin{equation}
  \label{eq:4}
  \leffib:\mnfld\to\cc,\quad   (\za,\zb,\zc)\mapsto \zc.
\end{equation}
The set of critical values of $\leffib$ is
\begin{equation}
  \label{eq:5}
  \critv=\sett{1,2,\ldots,n}.
\end{equation}

The fibers of $\leffib$ over non-critical values are isomorphic to $\cc^*$, and the fibers over the critical values are $\sett{\za\zb=0,\zc\in\critv}$.
The Hamiltonian flow of $\hamone$ preserves the Lefschetz fibers and foliates them into circles.
The Lagrangian torus fiber $L_{(b_1,b_2)}$ over $(b_1,b_2)\in B$ consists of the circles with $\hamone=b_1$ in all the Lefschetz fibers that lie over points $z$ in the circle $\sett{|z|^2= b_2}\subset\cc$ in the base of the Lefschetz fibration.

For the immersed Lagrangian spheres $\lagstd$,
\begin{displaymath}
  \pi(\lagstd)=\sett{|z|=r}
\end{displaymath}
and the singular point of $\lagstd$ coincides with the singular point of the Lefschetz fiber over $z=r$.
The intersection of $\lagstd$ and a smooth Lefschetz fiber is a circle which is a vanishing cycle for the singular point.

Near the singular point, the image under $\leffib$ of the $\p$-branch of $\lagstd$ is an arc that lies in the lower half-plane with one endpoint $z=r$, and the image of the $\q$-branch is an arc lying in the upper half-plane with one endpoint $z=r$.
\subsection{Special Lagrangian property and gradings}
Recall that $\poly=\za\zb-(\zc-1)\cdots(\zc-n)$ is the polynomial that cuts out $\mnfld$ and $\holvf$ is the Poincar\'e residue of $d\za\wedge d\zb \wedge d\zc/\poly\in\Omega^3_{\cc^3}(\mnfld)$.
More explicitly, $\holvf$ can be described as follows:
Given a point $p\in \mnfld$, choose any vector $V\in T_p\cc^3$ such that $d\poly(V)=1$.
Then at the point $p$
\begin{displaymath}
  \holvf(\cdot)=d\za\wedge d\zb\wedge d\zc(V,\cdot).
\end{displaymath}

\begin{lemma}
  \label{lemma:4}
  Let $L_b$ be a torus fiber.
  Then 
  \begin{displaymath}
    \mathrm{Im}\left(\frac{\holvf}{\zc}\right)\bigg|L_b=0.
  \end{displaymath}
  That is, $L_b$ is a special Lagrangian with respect to the meromorphic volume form $\holvf/\zc$.
\end{lemma}
\begin{proof}
  Fix a point $p\in\mnfld$ and let $V=(0,0,\alpha)\in T_p\cc^3$ with $\alpha=(\pd{\poly}{\zc})^{-1}$, so $d\poly(V)=1$.
  Then
  \begin{displaymath}
    \holvf=d\za\wedge d\zb\wedge d\zc(V,\cdot,\cdot)=\alpha d\za\wedge d\zb.
  \end{displaymath}
  Hence
  \begin{displaymath}
    \holvf(\hamvf{\hamone},\cdot)=-i \alpha  \za d\za -i \alpha  \zb d\za=-i\alpha d\poly+i\alpha\pd{\poly}{\zc}d\zc=i\alpha\pd{\poly}{\zc}d\zc=id\zc.
  \end{displaymath}
  Then
  \begin{displaymath}
    \frac{\holvf}{\zc}(\hamvf{\hamone},\hamvf{\hamtwo})=i\frac{d\zc}{\zc}(\hamvf{\hamtwo})=id\log \zc(\hamvf{\hamtwo}),
  \end{displaymath}
  and hence
  \begin{displaymath}
    \mathrm{Im}\left(\frac{\holvf}{\zc}(\hamvf{\hamone},\hamvf{\hamtwo})\right)=d\log|\zc|(\hamvf{\hamtwo})=0.
  \end{displaymath}
  Since $\hamvf{\hamone},\hamvf{\hamtwo}$ span the tangent space of the Lagrangian torus fibers, it follows that they are special Lagrangian with respect to $\holvf/\zc$.

\end{proof}

We briefly recall the notion of a grading in the context of almost Calabi-Yau manifolds.
Suppose $M$ is a symplectic manifold with almost complex structure $J$, $L$ is an embedded Lagrangian, and $\Omega$ is a nowhere vanishing section of $\Lambda^{top}(TM,J)$.
There is a canonical map $s_L:L\to \mathcal {G}$, where $\mathcal G$ is the bundle over $M$ whose fiber over a point $p$ is the set of Lagrangian planes in $T_pM$.
$\Omega$ determines a (squared) phase map 
\begin{displaymath}
  \mathrm{Det}^2_\Omega:\mathcal G\to S^1.
\end{displaymath}
Then a grading for $L$ is a map $\gr{L}:L\to\rr$ such that 
\begin{equation}
  \label{eq:3}
  e^{2\pi i\gr{L}}=\mathrm{Det}^2_\Omega\circ s_L.
\end{equation}

In case $L$ is only immersed, i.e.\ there is an immersion $i:L\to M$ with $i(L)$ Lagrangian, then a natural map $s_L:L\to\mathcal G$ still exists, and a grading is defined to be a function $\gr{L}:L\to\rr$ that satisfies \eqref{eq:3}.

Now consider the immersed Lagrangian spheres $\lagstd$.
Lemma \ref{lemma:4} shows that any function $\theta:S^2\to \rr$ that satisfies
\begin{displaymath}
  e^{2\pi i \theta}=\frac{\zc^2}{|\zc^2|}\circ\imm
\end{displaymath}
is a grading for the Lagrangian sphere.
We pick one such choice:

\begin{definition}
  \label{dfn:7}
  We grade $\imm:S^2\to\lagstd$ with respect to $\holvf$ by the function
  \begin{displaymath}
    \grstd(a,e^{ib})=\frac{f(a)}{\pi}.
  \end{displaymath}
  Here, $(a,e^{ib})$ are cylindrical coordinates on $S^2$ as in \eqref{eq:1}.
\end{definition}

The grading allows us to assign an index to a self-intersection point, or more precisely to pairs of branches (see the discussion following Definition \ref{dfn:3}).
Actually, because $S^2$ is connected, the index is independent of the choice of grading.
\begin{definition}
  \label{dfn:35}
  Let
  \begin{eqnarray*}
    \bjstd& = &\set{(p,q)\in S^2\times S^2}{\imm(p)=\imm(q),\ p\neq q}\\
    &=& \sett{(\p,\q),(\q,\p)}.
  \end{eqnarray*}
  For each element $(p,q)\in\bjstd$, we define an index by the formula
  \begin{displaymath}
    \ind (p,q)=n+\grstd(q)-\grstd(p)-2\cdot\ang(\immpf T_pS^2,\immpf T_q S^2).
  \end{displaymath}
  Here, $\ang(\immpf T_pS^2,\immpf T_qS^2)=a+b$ where
  \begin{displaymath}
    \immpf T_qS^2 =\left[
    \begin{array}{cc}
      e^{2\pi i a} & 0 \\
      0 & e^{2\pi i b}
    \end{array}
  \right]\cdot \immpf T_pS^2
  \end{displaymath}
  in an appropriate unitary basis.
\end{definition}

Equivalently, $\ind (p,q)$ can be defined in the following way:
In the Lagrangian Grassmannian, start with $T_qS^2$.
Move in the positive definite direction from $T_qS^2$ to $T_pS^2$, while simultaneously changing the real number $\grstd(q)$ to match the phase of the moving Lagrangian plane.
This will result in a new real number $\theta'$ when the moving plane reaches $T_pS^2$.
The index is then
\begin{displaymath}
  \ind (p,q)=\theta'-\grstd(p).
\end{displaymath}

\begin{lemma}
  \label{lemma:10}
  The indices of the elements of $\bjstd$ are
  \begin{eqnarray*}
    \ind (\p,\q) &= &-1,\\
    \ind (\q,\p) &=& 3.
  \end{eqnarray*}
\end{lemma}

The definitions and conventions given above agree with those in \cite{seidel-fcpclt}.

\subsection{Exactness}
Recall that $(\mnfld,\sympl,\oneform)$ is an exact symplectic manifold with $\sympl=d\lambda$.
Since $H^1(S^2)=0$, $\imm^*\oneform$ is exact.
Thus each $\lagstd$ is an exact immersed Lagrangian.
We fix a primitive $\exactprim:S^2\to\rr$ of $\imm^*\oneform$, so
\begin{displaymath}
  d\exactprim=\imm^*\oneform.
\end{displaymath}
The exact choice of $\exactprim$ is not important for our purposes.

\section{Discs and strips}
\label{sec:discs-strips}
In this section we discuss the notion of (boundary) marked discs and strips with boundary on $\lagstd$.
These discs and strips are allowed (but not required) to jump branches of $\lagstd$ only at marked points.

From a topological point of view there is not much difference between a disc and a strip.
However, they will have different roles to play in Floer cohomology.
Holomorphic strips will be used to define the Floer differential; we use strips because in general time dependent almost complex structures are needed to achieve transversality.
The discs play an auxiliary role---their main purpose is to help precisely describe the compactification of the moduli space of holomorphic strips.
Since the technical details of the variant of Floer cohomology used in this paper are relegated to \cite{alston-fciegl}, we do not really need discs, but we prefer to include an exposition for completeness.

A key point is that the Lagrangians $\lagstd$ are immersed so disc bubbles that connect via branch jumps can appear.
To help deal with this, we add extra structure to the definition of marked discs and strips:
Following \cite{MR2785840}, we include a lift of the boundary of the disc to $S^2$ as part of the data.

In this section we concentrate on the topology of discs and strips, including a discussion of the Maslov index.
Our conventions follow \cite{seidel-fcpclt}.
The analytic theory of discs and strips will be described in Sections \ref{sec:analytic-setup} and \ref{sec:moduli-spac-holom}.

\subsection{Marked discs}

\begin{definition}
  \label{dfn:9}
  A \textit{marked disc}
  \begin{displaymath}
    \disc{u}=(u,\dmp{\disc{u}},\ell)
  \end{displaymath}
  with boundary on $\imm:S^2\to\lagstd$ consists of the following data:
  \begin{itemize}
  \item A continuous map $u:(\dd,\bdy\dd)\to(\mnfld,\lagstd)$.
  \item A list (possibly empty) of marked boundary points $\dmp{\disc{u}}=(z_0,\ldots,z_k)$.
    Moreover, each marked point is labeled as incoming ($-$) or outgoing ($+$).
    We will generally suppress the labeling from the notation.
  \item A continuous map $\ell:\bdy\dd\setminus \dmp{\disc{u}}\to S^2$.
  \end{itemize}
  The maps are required to satisfy
  \begin{displaymath}
    \imm\circ \ell = u | \bdy \dd\setminus \dmp{\disc{u}}.
  \end{displaymath}
  If $C$ is a component of $\bdy \dd\setminus \dmp{\disc{u}}$, then we also require that $\ell$ extends continuously to $\overline{C}$.\footnote{In case $C=\bdy\dd\setminus\sett{pt}$, we think of $\overline{C}$ as being a closed interval instead of a circle.}
  We also define $\dmp{\disc{u}}^-=\set{z_i\in\dmp{\disc{u}}}{z_i\text{ is incoming}}$ and $\dmp{\disc{u}}^+=\set{z_i\in\dmp{\disc{u}}}{z_i\text{ is outgoing}}$.
\end{definition}

The map $\ell$ is discontinuous at a marked point if and only if $u$ switches branches of $\lagstd$ at the marked point.
We will refer to such behavior as a \textit{branch jump}, and call the marked point a \textit{branch jump point}, or \textit{branch point} for short.
Recall from Definition \ref{dfn:35} that
\begin{displaymath}
  \bjstd=\set{(p,q)\in S^2\times S^2}{\imm(p)=\imm(q),\ p\neq q}.
\end{displaymath}
$\bjstd$ can be thought of as the set of possible branch jump types.
More precisely, let the marked point $z_i\in\bdy \dd$ be a branch point.
As $z$ moves along $\bdy \dd$ in the counterclockwise direction towards $z_i$, $\ell(z)$ converges to a point $p\in S^2$.
Likewise, as $z$ moves along $\bdy\dd$ in the clockwise direction towards $z_i$, $\ell(z)$ converges to a point $q\in S^2$.
Since $z_i$ is a branch point, $p\neq q$.

\begin{definition}
  \label{dfn:2}
  With the above notation, if $z_i$ is an incoming point then the \textit{branch jump type} of $\disc{u}$ at $z_i$ is
  \begin{displaymath}
    (q,p)\in \bjstd;
  \end{displaymath}
  and if $z_i$ is an outgoing point then the \textit{branch jump type} is
  \begin{displaymath}
    (p,q)\in\bjstd.
  \end{displaymath}
\end{definition}

\begin{definition}
  \label{dfn:15}
  Given a marked disc $\disc{u}=(u,\dmp{\disc{u}},\ell)$, let
  \begin{displaymath}
    \brindices{\pm}{\disc{u}}=\set{i}{z_i\in\dmp{\disc{u}}^\pm\text{ is a branch point}}.
  \end{displaymath}
  Let
  \begin{displaymath}
    \brjumps{\pm}{\disc{u}} : \brindices{\pm}{\disc{u}} \to R
  \end{displaymath}
  be the function that assigns to $i\in \brindices{\pm}{\disc{u}}$ the branch jump type of the point $z_i$.
  We say that the branch jumps of $\disc{u}$ are of \textit{type} $\brjumps{\pm}{\disc{u}}$.
\end{definition}

\subsection{Marked strips}

Let $\str=\rr\times[0,1]$ and $\bdy\str=\rr\times\sett{0,1}$ with coordinates $s+it=(s,t)$.
Let $\bdyb\str=\rr\times\sett{0}$, $\bdyt\str=\rr\times\sett{1}$.
\begin{definition}
  \label{dfn:5}
  A \textit{marked strip} 
  \begin{displaymath}
    \strip{u}=(u,\smp{0,\strip{u}},\smp{1,\strip{u}},\ell)
  \end{displaymath}
  with boundary on $\imm:S^2\to\lagstd$ consists of the following data:
  \begin{itemize}
  \item A continuous map $u:(\str,\bdy\str)\to (\mnfld,\lagstd)$ such $\lim_{s\to\pm\infty}u(s,\cdot)=\text{constant}$, uniformly in $t$.
  \item Two lists $\smp{i,\strip{u}}=(z^i_1,\ldots,z^i_{k_i})\subset\bdy_i\str$, $i=0,1$ of marked boundary points.
    Moreover, each marked point is labeled as incoming or outgoing.
  \item A continuous map $\ell:\bdy\str\setminus \smp{0,\strip{u}}\cup\smp{1,\strip{u}}\to S^2$.
  \end{itemize}
  The maps are required to satisfy
  \begin{displaymath}
    \imm\circ\ell=u|\bdy\str\setminus \smp{0,\strip{u}}\cup\smp{1,\strip{u}}.
  \end{displaymath}
  For $i=0,1$ let
  \begin{eqnarray*}
    \smp{i,\strip{u}}^+&=&\set{z_j\in\smp{i,\strip{u}}}{z_j\text{ is outgoing}},\\
    \smp{i,\strip{u}}^-&=&\set{z_j\in\smp{i,\strip{u}}}{z_j\text{ is incoming}},\\
    \brindices{\pm}{i,\strip{u}}&=&\set{j}{z_j\in\smp{i,\strip{u}}^\pm\text{ is a branch point}}.
  \end{eqnarray*}
  Let $\brjumps{\pm}{i,\strip{u}}:\brindices{\pm}{i,\strip{u}}\to\bjstd$ denote the types of the branch jumps.
  Also, we view $s=-\infty$ as an additional incoming marked point and $s=+\infty$ as an additional outgoing marked point (although we do not include them in the lists $\smp{i,\strip{u}}^\pm$), and let $\brjumps{\pm\infty}{\strip{u}}\in \bjstd\amalg\sett{\emptyset}$ denote the branch jump types of $s=\pm\infty$ (if $s=\pm\infty$ is not a branch jump then $\brjumps{\pm\infty}{\disc{u}}=\emptyset$).
\end{definition}
\begin{remark}
  \label{rmk:5}
  Pick a biholomorphism $\phi:\dd\setminus\sett{-1,1}\to\str$.
  Then the marked strip $(u,\Sigma_0,\Sigma_1,\ell)$ naturally corresponds to the marked disc
  \begin{displaymath}
    (u\circ\phi,\Sigma=\Sigma_0\cup\Sigma_1\cup\sett{-\infty,+\infty},\ell\circ\phi).
  \end{displaymath}
  The only ambiguity is in the choice of the biholomorphism and the ordering of the marked points.
  Despite the ambiguity, many properties of marked discs carry over to give analogous properties of marked strips.
  For example, the notion of a branch jump and branch jump type carries over.
  We will take advantage of this similarity to avoid repeating similar definitions and lemmas.
  On the other hand, keep in mind that discs and strips will have different roles to play when we get to Floer theory.
\end{remark}

\subsection{Maslov index}
\label{sec:maslov-index-theory}
Recall that a bundle pair $(E,F)$ consists of the data
\begin{itemize}
\item a symplectic vector bundle $E\to \dd$,
\item a Lagrangian subbundle $F$ of $E|\bdy \dd$.
\end{itemize}
Let $\mu(E,F)$ denote the Maslov index of the bundle pair $(E,F)$.

Let $V$ be a symplectic vector space and $\Lambda_0$, $\Lambda_1$ two Lagrangian planes.
Let $J$ be a compatible almost complex structure such that $J\cdot\Lambda_0=\Lambda_1$.
The path of Lagrangian planes
\begin{displaymath}
  t\mapsto e^{\pi Jt/2}\cdot \Lambda_0,\qquad 0\leq t\leq 1
\end{displaymath}
is called the \textit{positive definite path} from $\Lambda_0$ to $\Lambda_1$.
It is well-defined up to homotopy.
Similarly, the path
\begin{displaymath}
  t\mapsto e^{-\pi Jt/2}\cdot\Lambda_0
\end{displaymath}
is the \textit{negative definite path} from $\Lambda_0$ to $\Lambda_1$.

Let $\disc{u}=(u,\dmp{\disc{u}},\ell)$ be a marked disc with boundary on $\lagstd$ as in Definition \ref{dfn:9}.
Associated to $\disc{u}$ we define a bundle pair as follows:
First, $E=u^*T\mnfld$.
Second, $F|\bdy \dd\setminus\dmp{\disc{u}}=\iota_*(\ell^*TS^2)$.
If a marked point is not a branch point, then $F$ extends over the marked point.
To extend $F$ over the branch points, we proceed as follows:
Homotope $F$ slightly so that it is constant on each side of the marked point.
This gives two Lagrangian planes; call the one that occurs before the marked point (in counterclockwise order) the first plane and the one that occurs after the marked point the second plane.
If the marked point is an incoming point then we extend $F$ by moving along the negative definite path from the  first plane to the second plane.
If the marked point is an outgoing point then we extend $F$ by moving along the positive definite path from the first plane to the second plane.

\begin{definition}
  \label{dfn:8}
  The \textit{Maslov index} $\mu(\disc{u})$ of a marked disc $\disc{u}$ is defined to be the Maslov index of the bundle pair $(E,F)$ constructed above.
\end{definition}

\begin{proposition}[\cite{seidel-fcpclt} Proposition 11.13]
  \label{prop:13}
  Let $\disc{u}$ be a marked disc.
  Then
  \begin{displaymath}
    \mu(\disc{u})=\sum_{i\in\brindices{-}{\disc{u}}} \ind \brjumps{-}{\disc{u}}(i)-\sum_{i\in\brindices{+}{\disc{u}}}\ind\brjumps{+}{\disc{u}}(i)+2(1-|\brindices{-}{\disc{u}}|).
  \end{displaymath}
\end{proposition}

\begin{remark}
  More generally, if the $2$ in the above formula is replaced by $n$, then formula holds for a graded immersed Lagrangian of dimension $n$.
  If $\dd$ is replaced by a Riemann surface $S$, then the $1$ needs to be replaced by $\chi(S)$.
\end{remark}

The Maslov index is additive in the sense that if an outgoing marked point of one disc is glued to an incoming marked point of another disc then the Maslov index of the glued disc is the sum of the Maslov indices of the two discs.

\begin{definition}
  \label{dfn:30}
  By Remark \ref{rmk:5}, a marked strip can be viewed as a marked disc, with the $-\infty$ end of the strip an incoming marked point and the $+\infty$ end an outgoing marked point.
  The \textit{Maslov index} $\mu(\strip{u})$ of the marked strip is defined to be the Maslov index of $\strip{u}$ as a marked disc.
\end{definition}

This agrees with the usual notion of the Maslov index of a strip.

\begin{proposition}
  \label{prop:14}
  Let $\disc{u}$ be a marked strip.
  If $-\infty$ is a branch jump let $\delta=1$, otherwise let $\delta=0$.
  Then
  \begin{eqnarray*}
    \mu(\disc{u})&=&\ind\brjumps{-\infty}{\strip{u}}-\ind\brjumps{+\infty}{\strip{u}}+
    \sum_{j=0,1;\ i\in\brindices{-}{j,\disc{u}}} \ind \brjumps{-}{j,\disc{u}}(i)-\sum_{j=0,1;\ i\in\brindices{+}{j,\disc{u}}}\ind\brjumps{+}{j,\disc{u}}(i)\\
    &&+2(1-\delta-|\brindices{-}{\strip{u}}|).
  \end{eqnarray*}
  Here, $\ind \brjumps{\pm\infty}{\strip{u}}$ is defined to be $0$ if $\pm\infty$ is not a branch jump.
\end{proposition}
\begin{proof}
  This follows from Proposition \ref{prop:13} by applying Remark \ref{rmk:5}.
  Recall also that $\brindices{-}{j,\strip{u}}$ include only branch points along top and bottom boundaries (i.e. not $\pm\infty$).
\end{proof}

\section{Analytic setup}
\label{sec:analytic-setup}
In this section we describe the Banach manifold structure on the space of discs and strips of a certain Sobolev regularity.
This will allow us, in Section \ref{sec:moduli-spac-holom}, to talk more precisely about holomorphic strips and discs, which are ultimately things we are interested in.

Much of the material in this section is standard so some of the details are skipped (full details will be given in \cite{alston-fciegl}).
Other references are \cite{seidel-fcpclt}, \cite{fooo} and \cite{MR2785840}.
However, since in this paper we are only concerned with Floer cohomology and these other references deal more generally with $A_\infty$-algebras, our point of view is a little different.
Let us explain this difference a little bit by comparing to \cite{seidel-fcpclt}.

In \cite{seidel-fcpclt}, Seidel first constructs the moduli space of discs with marked points, and then a universal bundle over it and then adds extra structure to it (perturbation data and strip like ends).
The holomorphic curves he studies are then the curves that have as domain one of the fibers in the universal bundle.
He describes the compactification of this moduli space, and then explains how transversality can be achieved.

In our setup, since we are only doing Floer cohomology, we only need strips (with no extra marked points).
However to show things are well-defined it is necessary to consider strips with extra marked points.
The reason is that by Gromov's compactness theorem a sequence of strips will degenerate to a (possibly broken) strip, possibly with some disc bubbles attached (a priori at branch points).
We can show that the broken strip with marked points (but not the disc bubbles) generically has the expected dimension, and then by the positivity assumption on the discs, this dimension must be negative and hence cannot exist (see \cite{alston-fciegl}).

Therefore the most important issue for us is to describe the space of strips with marked points, in particular transversality for strips with marked points, and it is towards this that our analytic setup is geared.
We include some details on the analytic setup for marked discs since this will be important for the purpose of describing Gromov compactification.
However, our methods allow us to not concern ourselves with the question of whether or not this space is a smooth manifold.

\subsection{Analytic setup of discs}
Suppose given the following data:
\begin{itemize}
\item An ordered list $\dmp{}=(z_0,\ldots,z_k)\subset\bdy\dd$ of boundary points such that the points are listed in counterclockwise cyclic order.
\item A decomposition $\dmp{}=\dmp{}^-\cup\dmp{}^+$ into incoming and outgoing points.
\item Sets $\brindices{\pm}{}\subset\set{i}{z_i\in\dmp{}^\pm}$ and functions $\brjumps{\pm}{}:\brindices{\pm}{}\to\bjstd$.
\end{itemize}
To simplify notation we abbreviate this data as $\metadmp,\metabrjumps$, and call it \textit{marked point data}.

\begin{definition}
  \label{dfn:31}
  A marked disc $\disc{u}=(u,\dmp{\disc{u}},\ell)$ \textit{has marked point data}
  \begin{displaymath}
    \metadmp=(\dmp{},\dmp{}^-,\dmp{}^+),\quad\metabrjumps=(\brjumps{-}{},\brjumps{+}{})
  \end{displaymath}
  if
  \begin{displaymath}
    \dmp{\disc{u}}=\dmp{},\quad \dmp{\disc{u}}^\pm=\dmp{}^\pm,\quad \brjumps{\pm}{\disc{u}}=\brjumps{\pm}{}.
  \end{displaymath}
\end{definition}

Suppose given $\metadmp$ as above.
Let $\dsurf{\metadmp}=\dd\setminus \dmp{}$.
Let $\strn=(-\infty,0)\times[0,1]$ and $\strp=(0,\infty)\times[0,1]$.
A choice of \textit{strip-like ends} for $\dsurf{\metadmp}$ consists of a biholomorphism
\begin{eqnarray*}
  \epsilon^-_i&:&\strn\to U_i^-\subset \dsurf{\metadmp},\text{ or}\\
  \epsilon^+_i&:&\strp\to U_i^+\subset \dsurf{\metadmp}\\
\end{eqnarray*}
for each point $z_i\in\dmp{}^\pm$.\footnote{If $z_i\in\dmp{}^-$ then we are given $\epsilon_i^-$, if $z_i\in\dmp{}^+$ we are given $\epsilon_i^+$.}
Here each $U_i^\pm$ is an open subset of $\dsurf{\metadmp}$ that is obtained by taking an open neighborhood in $\dd$ of $z_i$ and removing $z_i$, and $\epsilon^\pm_i$ is a biholomorphism onto $U_i^\pm$ such that $\lim_{s\to\pm\infty}\epsilon^\pm_i(s,\cdot)=z_i$.

\begin{definition}
  \label{dfn:18}
  Let $\disc{u}$ be a marked disc with marked point data $\metadmp,\metabrjumps$.
  Assume that $u$ is smooth, and constant near $\pm\infty$ on the strip like ends.
  Fix a choice of strip like ends for $\metadmp$, and a metric on $\mnfld$.\footnote{
  Since $\mnfld$ is not compact, $\wkp{\disc{u}}$ will depend on the choice of metric.
  To be explicit, we assume that we choose a metric that agrees with the standard metric coming from the embedding $\mnfld\subset\cc^3$ outside of some compact set.}
  For $\delta>0$ and $p>2$, we let
  \begin{displaymath}
    \wkp{\disc{u}}.
  \end{displaymath}
  denote the set of all sections $\xi$ over $\dsurf{\metadmp}$ of $u^*T\mnfld$ such that
  \begin{itemize}
  \item $\xi$ is in $\wkpcust_{loc}$,
  \item On each strip like end $\epsilon_i^\pm$,
    \begin{displaymath}
      \int_{\strpn}|\xi\circ\epsilon_i^\pm|^pe^{\delta|s|}dsdt<\infty.
    \end{displaymath}
  \item $\xi|\bdy\str$ lies in $\immpf(\ell^*T\lagstd)$.\footnote{What is meant here is that for $z\in\bdy\dd$, $\xi(z)\in \immpf(T_{\ell(z)}S^2)$.}
  \end{itemize}
\end{definition}

\begin{definition}
  \label{dfn:34}
  Given $\disc{u}$, define
  \begin{displaymath}
    \tldisc{\disc{u}}=\bigoplus_{i\notin \brindices{\pm}{\disc{u}} }T_{\ell(z_i)}S^2.
  \end{displaymath}
  Note that $\ell(z_i)$ is well-defined because $z_i$ is not a branch point.
\end{definition}
By the integrability over the strip-like ends and the Sobolev embedding theorem, a vector field $\xi\in\wkp{\disc{u}}$ extends continuously to $\dd$ and vanishes at the marked points.
Also, $\xi|\bdy\dsurf{\metadmp}$ defines a section $\xi_{bdy}$ of $\ell^*T\lagstd$.
The metric on $\mnfld$ induces a metric on $S^2$, and if $\lagstd$ is totally geodesic then exponentiation defines a map from $\wkp{\disc{u}}\oplus \tldisc{\disc{u}}$ to the space of marked discs with marked point data $\metadmp,\metabrjumps$.
The map is
\begin{displaymath}
  (\xi,V)\mapsto \exp_{\disc{u}}(\xi+\tilde V)=\disc{u_{\xi,V}}=(u_{\xi,V},\dmp{\disc{u_{\xi,V}}},\ell_{\xi,V}),
\end{displaymath}
where
\begin{itemize}
\item $V=(V_{i_1},\ldots,V_{i_m})\in\tldisc{\disc{u}}$, 
\item $\tilde V$ is a section of $u^*T\mnfld$ that agrees with parallel translations of $\immpf V_{i_j}$ along rays $t=\text{constant}$ near $\pm\infty$ in strip-like ends, and is $0$ far away from $\pm\infty$ (choosing cutoff functions on the strip-like ends gives a way of constructing $\tilde V$),
\item $u_{\xi,V}=\exp_{u}(\xi+\tilde V)$,
\item $\ell_{\xi,V}=\exp_{\ell}(\xi_{bdy}+\tilde V_{bdy})$, and
\item $\disc{u_{\xi,V}}$ has the same marked point data as $\disc{u}$.
\end{itemize}
We think of these types of marked discs as having regularity $\wkpcust$, and we will always restrict our attention to these types of discs.
\begin{definition}
  \label{dfn:32}
  Define
  \begin{displaymath}
    \spacemd(\metadmp,\metabrjumps)=\bigcup_{\disc{u}}\set{\exp_{\disc{u}}(\xi+\tilde V)}{\xi\in\wkp{\disc{u}},\ V\in \tldisc{\disc{u}}},
  \end{displaymath}
  where the union\footnote{Note that the union is not a disjoint union.} is over all smooth marked discs $\disc{u}$ that are constant near infinity on the strip-like ends.
\end{definition}

Now we want to allow the marked point data $\metadmp$ to vary.
\begin{definition}
  \label{dfn:33}
  Given $k\geq 0$, define $\conf(k+1)$ to be the set of all ordered lists
  \begin{displaymath}
    (z_0,\ldots,z_k)
  \end{displaymath}
  of distinct points in $\bdy\dd$ which are counterclockwise cyclically ordered, along with a labeling of incoming or outgoing for each point in the list.
  $\conf(k+1)$ can be identified with the set of all $\metadmp$ such that the underlying $\dmp{}$ has $k+1$ elements.
  We view $\conf(k+1)$ as a smooth manifold with several components; different labelings of any given list lie in different components.
\end{definition}

\begin{definition}
  \label{dfn:19}
  Given $\metabrjumps$, define
  \begin{displaymath}
    \spacemd_{k+1}(\metabrjumps)=\bigcup_{\metadmp}\spacemd(\metadmp,\metabrjumps).
  \end{displaymath}
  The union is over all $\metadmp\in\conf(k+1)$.
  We assume that $\metabrjumps$ is compatible with $\metadmp$ in the sense that $\metadmp$, $\metabrjumps$ is valid marked point data.
\end{definition}

\begin{proposition}
  \label{prop:12}
  $\spacemd(\metadmp,\metabrjumps)$ has the structure of a smooth Banach manifold such that
  \begin{displaymath}
    T_{\disc{u}}\spacemd(\metadmp,\metabrjumps)\cong \wkp{\disc{u}}\oplus \tldisc{\disc{u}}.
  \end{displaymath}
  $\spacemd_{k+1}(\metabrjumps)$ has the structure of a $C^0$-Banach manifold.
  Locally it is modeled on open neighborhoods of $0$ in the Banach space
  \begin{displaymath}
  T_{\disc{u}}\spacemd(\metadmp,\metabrjumps)\oplus T_{\metadmp{}}\conf(k+1).
  \end{displaymath}
\end{proposition}

\subsection{Analytic setup for strips}
Suppose given the following data:
\begin{itemize}
\item Ordered lists $\dmp{i}=(z^i_1,\ldots,z^i_{k_i})\subset\bdy_i\str$ of boundary points such that the points are listed in counterclockwise cyclic order.
\item A decomposition $\dmp{i}=\dmp{i}^-\cup\dmp{i}^+$ into incoming and outgoing points.
\item Sets $\brindices{\pm}{i}\subset\set{i}{z_i\in\dmp{}^\pm}$ and functions $\brjumps{\pm}{i}:\brindices{\pm}{i}\to\bjstd$.
\item Elements $\alpha^{\pm\infty}\in \bjstd\cup\sett{\emptyset}$.
\end{itemize}
To simplify notation we abbreviate this data as $\metadmp,\metabrjumps$, and call it \textit{marked point data for a strip}.

\begin{definition}
  \label{dfn:6}
  A marked strip $\strip{u}=(u,\smp{0,\strip{u}},\smp{1,\strip{u}},\ell)$ \textit{has marked point data}
  \begin{displaymath}
    \metadmp=(\smp{0},\smp{1},\smp{0}^\pm,\smp{1}^\pm),\quad\metabrjumps=(\brjumps{\pm}{0},\brjumps{\pm}{1},\brjumps{\pm\infty}{})
  \end{displaymath}
  if
  \begin{displaymath}
    \smp{i,\disc{u}}=\smp{i},\quad \smp{i,\strip{u}}^\pm=\smp{i}^\pm,\quad \brjumps{\pm}{i,\strip{u}}=\brjumps{\pm}{i},\quad \brjumps{\pm\infty}{\disc{u}}=\brjumps{\pm\infty}{}\quad i=0,1.
  \end{displaymath}
\end{definition}

Suppose given $\metadmp$ as above.
Recall that $\ssurf{\metadmp}=\str\setminus (\dmp{0}\cup\dmp{1})$.
Let $\strn=(-\infty,0)\times[0,1]$ and $\strp=(0,\infty)\times[0,1]$.
A choice of \textit{strip-like ends} for $\ssurf{\metadmp}$ consists of a biholomorphism
\begin{eqnarray*}
  \epsilon^-_i&:&\strn\to U_i^-\subset \ssurf{\metadmp},\text{ or}\\
  \epsilon^+_i&:&\strp\to U_i^+\subset \ssurf{\metadmp}\\
\end{eqnarray*}
for each point $z^i_j\in\dmp{i}^\pm$, with the same properties as in the case of discs.
We also view $-\infty$ as an incoming marked point, and $+\infty$ as an outgoing marked point, both with the obvious strip like end structures (i.e.\ coming from $\str$ itself).

\begin{definition}
  \label{dfn:37}
  Let $\strip{u}$ be a marked strip with marked point data $\metadmp,\metabrjumps$.
  Assume that $u$ is smooth, and constant near $\pm\infty$ on the strip like ends.
  Fix a choice of strip like ends for $\metadmp$, and a metric on $\mnfld$.
  For $\delta>0$ and $p>2$, we let
  \begin{displaymath}
    \wkp{\strip{u}}.
  \end{displaymath}
  denote the set of all sections $\xi$ over $\ssurf{\metadmp}$ of $u^*T\mnfld$ such that
  \begin{itemize}
  \item $\xi$ is in $\wkpcust_{loc}$.
  \item On each strip like end $\epsilon_i^\pm$,
    \begin{displaymath}
      \int_{\strpn}|\xi\circ\epsilon_i^\pm|^pe^{\delta|s|}dsdt<\infty.
    \end{displaymath}
  \item For $R$ large enough\footnote{More precisely: $\sett{|s|>R}$ does not contain any marked points.}
    \begin{displaymath}
      \int_{|s|>R}|\xi|^pe^{\delta|s|}dsdt<\infty.
    \end{displaymath}
  \item $\xi|\bdy\ssurf{\metadmp}$ lies in $\immpf(\ell^*T\lagstd)$.
  \end{itemize}
\end{definition}

\begin{definition}
  \label{dfn:36}
  Given $\disc{u}$ a strip, define
  \begin{displaymath}
    \tldisc{\disc{u}}=W^{-\infty}\oplus W^{+\infty}\oplus\bigoplus_{i\notin \brindices{\pm}{0,\disc{u}} }T_{\ell(z^0_i)}S^2\oplus\bigoplus_{i\notin \brindices{\pm}{1,\disc{u}} }T_{\ell(z^1_i)}S^2,
  \end{displaymath}
  where
  \begin{displaymath}
    W^{\pm\infty}=\left\{
      \begin{array}{ll}
        \sett{0} & \brjumps{\pm\infty}{\strip{u}}\in\bjstd,\\
        T_{\ell(\pm\infty)}S^2 & \brjumps{\pm\infty}{\strip{u}}=\emptyset.
      \end{array}\right.
  \end{displaymath}
\end{definition}

As before, exponentiation defines a map from $\wkp{\strip{u}}\oplus \tldisc{\strip{u}}$ to the space of marked strips with marked point data $\metadmp,\metabrjumps$.
The map is
\begin{displaymath}
  (\xi,V)\mapsto \exp_{\disc{u}}(\xi+\tilde V)=\strip{u_{\xi,V}}=(u_{\xi,V},\dmp{\disc{u_{\xi,V}}},\ell_{\xi,V}),
\end{displaymath}
where
\begin{itemize}
\item $V=(V_{-\infty},V_{+\infty},V^0_{i_1},\ldots,V^1_{i_m})\in\tldisc{\disc{u}}$, 
\item $\tilde V$ is a section of $u^*T\mnfld$ that agrees with parallel translations of $\immpf V_{i_j}$ along rays $t=\text{constant}$ near $\pm\infty$ in strip-like ends, and is $0$ far away from $\pm\infty$ (choosing cutoff functions on the strip-like ends gives a way of constructing $\tilde V$),
\item $u_{\xi,V}=\exp_{u}(\xi+\tilde V)$,
\item $\ell_{\xi,V}=\exp_{\ell}(\xi_{bdy}+\tilde V_{bdy})$, and
\item $\disc{u_{\xi,V}}$ has the same marked point data as $\disc{u}$.
\end{itemize}

\begin{definition}
  \label{dfn:1}
  Define
  \begin{displaymath}
    \spacemd(\metadmp,\metabrjumps)=\bigcup_{\disc{u}}\set{\exp_{\disc{u}}(\xi+\tilde V)}{\xi\in\wkp{\disc{u}},\ V\in \tldisc{\disc{u}}},
  \end{displaymath}
  where the union is over all smooth marked strips $\disc{u}$ that are constant near infinity on the strip-like ends.
\end{definition}

\begin{definition}
  \label{dfn:11}
  Given $k_0,k_1\geq 0$, define $\conf(k_0,k_1)$ to be the set of all pairs of ordered lists
  \begin{displaymath}
    (z^0_1,\ldots,z^0_{k_0}),\ (z^1_1,\ldots,z^1_{k_1})
  \end{displaymath}
  of distinct points in $\bdy_0\str,\bdy_1\str$ which are counterclockwise cyclically ordered, along with a labeling of incoming or outgoing for each point in the list.
  $\conf(k_0,k_1)$ can be identified with the set of all $\metadmp$ such that the underlying $\dmp{i}$ have $k_i$ elements.
  We view $\conf(k_0,k_1)$ as a smooth manifold with several components; different labelings of any given list lie in different components.
\end{definition}

\begin{definition}
  \label{dfn:14}
  Given $\metabrjumps$, define
  \begin{displaymath}
    \spacemd_{k_0,k_1}(\metabrjumps)=\bigcup_{\metadmp}\spacemd(\metadmp,\metabrjumps).
  \end{displaymath}
  The union is over all $\metadmp\in\conf(k_0,k_1)$.
  We assume that $\metabrjumps$ is compatible with $\metadmp$ in the sense that $\metadmp$, $\metabrjumps$ is valid marked point data.
\end{definition}

\begin{proposition}
  \label{prop:2}
  $\spacemd(\metadmp,\metabrjumps)$ has the structure of a smooth Banach manifold such that
  \begin{displaymath}
    T_{\disc{u}}\spacemd(\metadmp,\metabrjumps)\cong \wkp{\disc{u}}\oplus \tldisc{\disc{u}}.
  \end{displaymath}
  $\spacemd_{k_0,k_1}(\metabrjumps)$ has the structure of a $C^0$-Banach manifold.
  Locally it is modeled on the Banach space
  \begin{displaymath}
    T_{\disc{u}}\spacemd(\metadmp,\metabrjumps)\oplus T_{\metadmp{}}\conf(k_0,k_1).
  \end{displaymath}
\end{proposition}

\section{Moduli spaces of holomorphic curves}
\label{sec:moduli-spac-holom}
\subsection{Holomorphic discs}
Let $J$ be a compatible complex structure which agrees with $\acs$ outside of a compact set.
\begin{definition}
  \label{dfn:38}
  For $k\geq 0$, let
  \begin{displaymath}
    \parmodmdstd
  \end{displaymath}
  denote the set of all pairs $(\disc{u},\metadmp)$ such that $\disc{u}\in\spacemd(\metadmp,\metabrjumps)$ and $u$ satisfies
  \begin{displaymath}
    du+J\circ du\circ j_\dd=0,
  \end{displaymath}
  and $|\metadmp|=k$ (i.e.\ there are $k$ marked points).
  In case $k\leq 2$ we require that $\disc{u}$ has non-constant $u$.
  Let
  \begin{displaymath}
    \modmdstd=\parmodmdstd/\Aut(\dd,\bdy\dd).
  \end{displaymath}
  These spaces have natural evaluation maps
  \begin{displaymath}
    \ev_j:\modmdstd\to S^2\amalg\bjstd.
  \end{displaymath}
\end{definition}

We now describe the Fredholm theory of these moduli spaces.
Let 
\begin{displaymath}
  \bbundle(\metadmp,\metabrjumps)\to\spacemd(\metadmp,\metabrjumps)
\end{displaymath}
be the smooth Banach bundle over $\spacemd(\metadmp,\metabrjumps)$ whose fiber over $\disc{u}$,
\begin{displaymath}
  \bbundle(\metadmp,\metabrjumps)_{\disc u}=\lp(\Lambda^{0,1}\otimes\disc{u}^*TM),
\end{displaymath}
consists of all $L^p_{loc}$ sections that are also in $L^p$ over the strip-like ends.
This bundle comes equipped with the smooth section
\begin{displaymath}
  \dbar_{\metadmp,\metabrjumps,J}:\spacemd(\metadmp,\metabrjumps)\to\bbundle(\metadmp,\metabrjumps),\quad u\mapsto du+J\circ du \circ j_{\dd}.
\end{displaymath}

Likewise, let
\begin{displaymath}
  \bbundle_k(\metabrjumps)\to\spacemd_k(\metabrjumps)
\end{displaymath}
be the $C^0$-Banach bundle over $\spacemd_k(\metabrjumps)$ defined in a similar way.
It comes equipped with the continuous section
\begin{displaymath}
  \dbar_{k,\metabrjumps,J}:\spacemd_k(\metabrjumps)\to\bbundle_k(\metabrjumps).
\end{displaymath}
Restricted to the local coordinate charts referenced in Proposition \ref{prop:12}, the Banach bundles have trivializations in which the section $\dbar_{k,\metabrjumps}$ becomes smooth.

\begin{proposition}
  \label{prop:15}
  The moduli space of (parameterized) holomorphic discs equals the zero set of the section $\dbar_{k,\metabrjumps,J}$,
  \begin{displaymath}
   \parmodmdstd=\dbar_{k,\metabrjumps,J}^{-1}(0) .
  \end{displaymath}
\end{proposition}

\begin{proposition}
  \label{prop:17}
  Let $\delta>0$ be sufficiently small.
  For $(\disc{u},\metadmp)\in\parmodmdstd$, local trivializations of the Banach manifold and bundle can be chosen so that $\dbar_{k,\metabrjumps}$ is smooth and the linearization
  \begin{displaymath}
    D_{\disc{u},\metadmp}\dbar_{k,\metabrjumps,J}:\wkp{\disc{u}}\oplus \tldisc{\disc{u}} \oplus T_{\metadmp}\conf(k)\to \lp(\Lambda^{0,1}\otimes\disc{u}^*T\mnfld)
  \end{displaymath}
  is Fredholm.
  The index is
  \begin{eqnarray*}
    \ind D_{\disc{u},\metadmp}\dbar_{k,\metabrjumps,J} &=& \mu(\disc{u})+2(1-\brindices{+}{\disc{u}})+k\\
    &=&\sum_{i\in\brindices{-}{\disc{u}}}\ind\brjumps{-}{\disc{u}}(i)-\sum_{i\in\brindices{+}{\disc{u}}}\ind\brjumps{+}{\disc{u}}(i)+2(1-\brindices{-}{\disc{u}})+k.
  \end{eqnarray*}
\end{proposition}

\subsection{Holomorphic strips}
The story for holomorphic strips is very similar; the main difference is that we consider a $t$-dependent complex structure
\begin{displaymath}
  \J=\jstd.
\end{displaymath}

\begin{definition}
\label{dfn:39}
For $k_0,k_1\geq0$, let
\begin{displaymath}
  \parmodmsstd
\end{displaymath}
denote the set of all pairs $(\disc{u},\metadmp)$ such that $\disc{u}\in\spacemd(\metadmp,\metabrjumps)$, $\metadmp\in\conf(k_0,k_1)$ and
\begin{displaymath}
  du+J_t\circ du\circ j_\str=0.
\end{displaymath}
In case $k_0=k_1=0$, we require that $\disc{u}$ has non-constant $u$.
Also, define
\begin{displaymath}
  \modmsstd=\parmodmsstd/\rr.
\end{displaymath}
These spaces have natural evaluation maps
\begin{displaymath}
  \ev_j,\ev_{\pm\infty}:\modmsstd\to S^2\amalg\bjstd.
\end{displaymath}
\end{definition}

Banach bundles and $\dbar$-sections can be constructed in the same way as for discs.
This leads to smooth bundle
\begin{displaymath}
  \bbundle(\metadmp,\metabrjumps)\to\spacemd(\metadmp,\metabrjumps)
\end{displaymath}
and smooth section
\begin{displaymath}
  \dbar_{\metadmp,\metabrjumps,\J}:\spacemd(\metadmp,\metabrjumps)\to\bbundle(\metadmp,\metabrjumps),\quad u\mapsto du+J_t\circ du \circ j_{\str}.
\end{displaymath}
Also there is a $C^0$-bundle
\begin{displaymath}
  \bbundle_{k_0,k_1}(\metabrjumps)\to\spacemd_{k_0,k_1}(\metabrjumps)
\end{displaymath}
and $C^0$-section
\begin{displaymath}
  \dbar_{k_0,k_1,\metabrjumps,\J}:\spacemd_{k_0,k_1}(\metabrjumps)\to\bbundle_{k_0,k_1}(\metabrjumps).
\end{displaymath}

The analog of Proposition \ref{prop:17} is
\begin{proposition}
  \label{prop:16}
  Let $\delta>0$ be sufficiently small.
  For $(\disc{u},\metadmp)\in\parmodmsstd$, local trivializations of the Banach manifold and bundle can be chosen so that $\dbar_{k_0,k_1,\metabrjumps,\J}$ is smooth and the linearization
  \begin{displaymath}
    D_{\disc{u},\metadmp}\dbar_{k_0,k_1,\metabrjumps,\J}:\wkp{\disc{u}}\oplus \tldisc{\disc{u}} \oplus T_{\metadmp}\conf(k_0,k_1)\to \lp(\Lambda^{0,1}_{\ssurf{\metadmp}}\otimes\disc{u}^*TM)
  \end{displaymath}
  is Fredholm.
  The index is
  \begin{eqnarray*}
    \ind D_{\disc{u},\metadmp}\dbar_{k_0,k_1,\metabrjumps,\J} &=& \mu(\disc{u})-2|\brindices{+}{\disc{u}}|+k_0+k_1\\
    &=&\ind\brjumps{-\infty}{\strip{u}}-\ind\brjumps{+\infty}{\strip{u}}\\
    &&+\sum_{j=0,1;\ i\in\brindices{-}{j,\disc{u}}} \ind \brjumps{-}{j,\disc{u}}(i)-\sum_{j=0,1;\ i\in\brindices{+}{j,\disc{u}}}\ind\brjumps{+}{j,\disc{u}}(i)\\
    &&+2(1-\delta-|\brindices{-}{\strip{u}}|)+k_0+k_1.
  \end{eqnarray*}
  Here $\delta=1$ if $-\infty$ is a branch point, otherwise $\delta=0$.
\end{proposition}

\section{Classifying holomorphic curves}
\label{sec:class-holom-curv}
In this section we classify all holomorphic discs and strips in $\mnfld$ with boundary on $\lagstd$.
The complex structure we use is the standard one $\acs$.
We also prove that all the discs and strips are regular.

\subsection{Classifying discs}
Consider the Lefschetz fibration
\begin{displaymath}
  \leffib:\mnfld\to\cc,\quad (\za,\zb,\zc)\mapsto \zc.
\end{displaymath}
The key fact is the following:
If $u=(f,g,h)$ is a holomorphic disc with boundary on $\lagstd$, then $\pi\circ u=h$ is a holomorphic disc with boundary on $\pi(\lagstd)=\sett{z=r}$.
By the maximum principle, if $u$ has a branch jump at some marked point, and the marked point is outgoing, then the branch jump must be of type $(\p,\q)$.\footnote{If the marked point is incoming then the type must be $(\q,\p)$.}
See Definition \ref{dfn:2} and the discussion at the end of Section \ref{sec:lefsch-fibr-persp}.

We first consider discs on $\lag{N,1}\subset\mnfldspec{N}$.
By exactness, there are no non-constant discs without branch jumps.
Suppose that $\disc{u}=(u,\dmp{\disc{u}},\ell)$ is a marked disc with
\begin{displaymath}
  \dmp{\disc{u}}=\dmp{\disc{u}}^+=(z_0,\ldots,z_{k}),\quad \brjumps{+}{\disc{u}}(i)=(\p,\q)\ \forall i.
\end{displaymath}

\begin{lemma}
  \label{lemma:13}
  For $\disc{u}=(u,\dmp{\disc{u}},\ell)$ as above, $u$ must be of the form
  \begin{displaymath}
    u:z\mapsto (e^{i\theta}\xi,e^{-i\theta}\xi,h)
  \end{displaymath}
  where
  \begin{displaymath}
    \xi=\sqrt{h-1}\sqrt{h-2}\cdots\sqrt{h-N}
  \end{displaymath}
  and $h:\dd\to\dd$ is a Blaschke product of the form
  \begin{displaymath}
    h:z\mapsto \lambda\prod_{j=0}^{k} \frac{z-\alpha_j}{\bar\alpha_j z -1}
  \end{displaymath}
  with $|\lambda|=1$, $|\alpha_j|<1$ and
  \begin{displaymath}
    h^{-1}(1)=\sett{z_0,\ldots,z_{k}}.
  \end{displaymath}
  To be precise, we define the square root function by taking a branch cut along the positive real axis and letting $\sqrt{-1}=i$.
\end{lemma}
\begin{proof}
  Let $h=\pi\circ u$.
  Then $h:\dd\to\dd$ is holomorphic and hence must be a Blaschke product.
  In particular, when restricted to $\bdy\dd$, $h$ induces a map $S^1\to S^1$ that is strictly increasing, and hence the winding number is equal to the number of points in the preimage of any point.
  Branch jumps occur precisely at the points $z$ on the boundary where $h(z)=1$, which proves
  \begin{displaymath}
    h^{-1}(1)=\sett{z_0,\ldots,z_k}.
  \end{displaymath}

  Now suppose $u=(f,g,h)$.
  It remains to prove that
  \begin{displaymath}
    f=e^{i\theta}\xi,\quad g=e^{-i\theta}\xi
  \end{displaymath}
  for some $\theta$.
  Note that $\xi$ is holomorphic and non-zero on the interior of $\dd$ and continuous up to the boundary.
  Let $F=f/\xi$, so $F$ is a holomorphic function on the interior of $\dd$ and extends continuously to $\dd\setminus\sett{z_0,\ldots,z_k}$.
  The defining equation of $\mnfld$ implies
  \begin{displaymath}
    F=\frac{f}{\xi}=\frac{\xi}{g}.
  \end{displaymath}
  The second equality implies that $F$ has no zeroes on the interior of $\dd$.
  Also, by the Lagrangian boundary conditions, $|f|=|g|$ on the boundary of $\dd$, and hence $|F|=1$ on the boundary.
  $F$ can be viewed as a holomorphic map $\dd\setminus\sett{z_0,\ldots,z_{k}}\to\cc$ with boundary in the Lagrangian $S^1\to\cc$.
  By the removable singularities theorem, $F$ must extend to smoothly to a map $\dd\to\dd$.
  Coupled with the fact that $F$ has no zeroes on the interior, this implies that $F$ is constant, and hence $F=e^{i\theta}$ for some $\theta$.
  Thus
  \begin{displaymath}
    u=(f,g,h)=(e^{i\theta}\xi,e^{-i\theta}\xi,h).
  \end{displaymath}
\end{proof}

More generally, we have
\begin{proposition}
  \label{prop:4}
  Let $\disc{u}$ be a holomorphic disc with boundary on $\lag{N,r}\subset\mnfld$ with $r\in\sett{1,\ldots,N}$.
  Assume that all marked points are branch points.
  Write
  \begin{displaymath}
    \dmp{\disc{u}}=\dmp{\disc{u}}^+=\brindices{+}{\disc{u}}=(z_0,\ldots,z_k),\quad \brjumps{+}{\disc{u}}(i)=(\p,\q)\ \forall i.
  \end{displaymath}
  Then $u=(f,g,rh)$, where $h$ is a Blaschke product such that
  \begin{displaymath}
    h^{-1}(1)=\sett{z_0,\ldots,z_k}
  \end{displaymath}
  and
  \begin{eqnarray*}
    f&=&e^{i\theta}\xi_1\xi_2\cdots\xi_{r-1}\sqrt{h-r}\sqrt{2h-r}\cdots\sqrt{rh-r}\sqrt{rh-(r+1)}\cdots\sqrt{rh-N},\\
    g&=&e^{-i\theta}\eta_1\eta_2\cdots\eta_{r-1}\sqrt{h-r}\sqrt{2h-r}\cdots\sqrt{rh-r}\sqrt{rh-(r+1)}\cdots\sqrt{rh-N}.\\
  \end{eqnarray*}
  Here $\xi_j$ and $\eta_j$ are Blaschke products that satisfy
  \begin{displaymath}
    \xi_j\eta_j=\frac{rh-j}{jh-r}.
  \end{displaymath}
  We define the square root function by taking a branch cut along the positive real axis and setting $\sqrt{-1}=i$.
\end{proposition}
\begin{proof}
  The proof is very similar to the proof of the previous lemma.
  First, the above $u$ is a holomorphic disc on $\lagstd$ with the stated properties.
  Second, suppose $u=(f,g,rh)$ is any disc satisfying the hypotheses of the proposition.
  Let $u_0=(f_0,g_0,rh)$ with
  \begin{eqnarray*}
    f_0&=&\xi_1\cdots\xi_{r-1}\sqrt{h-r}\sqrt{2h-r}\cdots\sqrt{rh-r}\sqrt{rh-(r+1)}\cdots\sqrt{rh-N},\\
    g_0&=&\sqrt{h-r}\sqrt{2h-r}\cdots\sqrt{rh-r}\sqrt{rh-(r+1)}\cdots\sqrt{rh-N},\\
    \xi_j&=&\frac{rh-j}{jh-r}.
  \end{eqnarray*}
  Write $f=Ff_0$, $g=\frac{1}{F}g_0$.
  Then it is easy to see that $|F|=1$ on the boundary of $\dd$, and $1/F$ is a Blaschke product dividing $\xi_1\cdots \xi_{r-1}$.
  The proposition follows.
\end{proof}

As a sanity check, let us calculate the dimension of the space of marked discs using the previous proposition.
$h$ is a Blaschke product of order $k+1$, hence can be written uniquely in the form
\begin{displaymath}
  h=\lambda\prod_{j=0}^k\frac{z-\alpha_j}{\bar\alpha_j z-1}
\end{displaymath}
with $\lambda\in S^1$ and $|\alpha_j|<1$.
Note that $\lambda$ is determined by the condition that $h(z_0)=1$.
Thus the choice of $h$ contributes $2(k+1)=2k+2$ to the dimension.
Next, choosing different $\xi_j$ and $\eta_j$ such that $\xi_j\eta_j=(rh-j)/(jh-r)$ just changes the component of the moduli space.
The only remaining choice is changing $e^{i\theta}$.
Thus the dimension is
\begin{displaymath}
  2k+3.
\end{displaymath}
By Proposition \ref{prop:17}, we also get that the dimension is
\begin{displaymath}
  -(k+1)(-1)+2+k=2k+3.
\end{displaymath}

The next corollary follows from the removable singularities theorem.
\begin{corollary}
  \label{cor:2}
  Any marked disc is obtained by taking a disc as in Proposition \ref{prop:4} and then adding some non-branch marked points to it.
\end{corollary}

\subsection{Classifying strips}
\label{sec:classifying-strips}

Let $J_t=\acs$ for all $t$ and $\J=\depj{t}{t}$.
Then, using Remark \ref{rmk:5}, we can think of a holomorphic strip as a holomorphic disc.
Thus the results of the previous section can be used to classify all holomorphic strips.
For example, we have
\begin{proposition}
  \label{prop:5}
  Let $\strip{u}=(u,\dmp{0,\strip{u}},\dmp{1,\strip{u}},\ell)$ with $\dmp{0,\strip{u}}=\dmp{1,\strip{u}}=\emptyset$ be a holomorphic strip with boundary on $\lagstd$.
  Then $\strip{u}$ has $1$ or $2$ branch jumps; if $-\infty$ is a branch jump it is of type $(\q,\p)$ and if $+\infty$ is a branch jump it is of type $(\p,\q)$.
  Moreover, there exists a marked disc $\disc{v}$ with $\dmp{\disc{v}}=\sett{z_0,z_1}$ and a biholomorphism $\phi:\str\to\dd\setminus\sett{z_0,z_1}$ such that $u=v\circ\phi$.
  The converse is also true in the sense that any such $v\circ\phi$ gives rise to a holomorphic strip.
  (However, there may be many ways to write any given $u$ in such a way.)
\end{proposition}

\subsection{Regularity}
We prove that the discs in Proposition \ref{prop:4} are regular.
Let $\disc{u}$ be a disc from the proposition, with $k+1$ marked boundary points each of which is a branch jump of type $(\p,\q)$ (thinking of the marked points as outgoing).
Let $\metadmp,\metabrjumps$ be the marked point data for $\disc{u}$, so $\disc{u}\in\spacemd(\metadmp,\metabrjumps)$.
To show that $\disc{u}$ is regular we need to show that the linearization of
\begin{displaymath}
  \dbar_{k+1,\metabrjumps,J}:\spacemd_{k+1}(\metabrjumps)\to\bbundle_{k+1}(\metabrjumps)
\end{displaymath}
at $(\disc{u},\metadmp)$ is surjective.
The linearization is the map
\begin{displaymath}
  D_{\disc{u},\metadmp}\dbar_{k+1,\metabrjumps,J}:\wkp{\disc{u}}\oplus\tldisc{\disc{u}}\oplus T_{\metadmp}\conf(k+1)\to \lp(\Lambda^{0,1}_{\dsurf{\metadmp}}\otimes\disc{u}^*T\mnfld).
\end{displaymath}
Clearly it suffices to show that the linearization of
\begin{displaymath}
  \dbar_{\metadmp,\metabrjumps,J}:\spacemd(\metadmp,\metabrjumps)\to\bbundle(\metadmp,\metabrjumps)
\end{displaymath}
at $\disc{u}$ is surjective, i.e.\ that we can get surjectivity without varying the marked points.
The linearization of this operator is the map
\begin{displaymath}
  D_{\disc{u}}\dbar_{\metadmp,\metabrjumps,J}:\wkp{\disc{u}}\oplus\tldisc{\disc{u}}\to\lp(\Lambda^{0,1}_{\dsurf{\metadmp}}\otimes\disc{u}^*T\mnfld).
\end{displaymath}
All marked points are branch jumps so $\tldisc{\disc{u}}=\sett{0}$.
Since $(\mnfld,\sympl,\acs)$ is K\"ahler, the linearization is the Dolbeault operator.
Thus it suffices to show that the Dolbeault operator
\begin{displaymath}
  \dbar:\wkp{\disc{u}}\to\lp(\Lambda^{0,1}_{\dsurf{\metadmp}}\otimes\disc{u}^*T\mnfld)
\end{displaymath}
is surjective.

We can write $u=(e^{i\theta}\xi\sqrt{rh-r}F,e^{-i\theta}\eta\sqrt{rh-r}F)$ where $F$ is a non-vanishing holomorphic function on the disc, $\xi=\xi_1\cdots\xi_{r-1}$, $\eta=\eta_1\cdots\eta_{r-1}$, and $h$ is a Blaschke product with $h^{-1}(1)=\sett{z_0,\ldots,z_k}$.
Thinking of $u$ as a family of discs, differentiating with respect to $\theta$ gives
\begin{displaymath}
  \pd{u}{\theta}=(ie^{i\theta}\xi\sqrt{rh-r}F,-ie^{-i\theta}\eta\sqrt{rh-r}F,0).
\end{displaymath}
This is a holomorphic section of $u^*T\mnfld$.
It follows that 
\begin{displaymath}
  S=(ie^{i\theta}\xi F,-ie^{-i\theta}\eta F,0)
\end{displaymath}
 is also a holomorphic section and hence defines a holomorphic line bundle
\begin{displaymath}
  \linebundle\subset u^*T\mnfld.
\end{displaymath}
Define Lagrangian boundary conditions $\lambda$ for this line bundle by
\begin{displaymath}
  \lambda(z)=\sqrt{rh-r}\rr\cdot S(z),\quad z\in\bdy\dsurf{\metadmp}.
\end{displaymath}
Since $\linebundle$ is holomorphic, we get a commutative diagram of operators
\begin{displaymath}
  \begin{array}{ccccc}
    \wkpcust(\linebundle,\lambda) & \to & \wkp{\disc{u}} & \to & X\\
    \downarrow &&\downarrow &&\downarrow\\
    \lp(\Lambda^{0,1}\otimes\linebundle) &\to & \lp(\Lambda^{0,1}\otimes\disc{u}^*T\mnfld) &\to &Y.
  \end{array}
\end{displaymath}
Here $X$ and $Y$ are quotients so that the rows are exact.
The first two vertical arrows are Dolbeault operators and the last vertical arrow is the operator induced by the second vertical arrow.
It is well-defined because $\linebundle$ is holomorphic.

Thinking of all the marked points as outgoing, the Maslov index of the first operator is $k+1$.
This is because as $z$ varies on an arc of $\bdy\dsurf{\metadmp}$ from the marked point $z_j$ to $z_{j+1}$, the phase of $\sqrt{rh-r}$ varies from $e^{i\pi/4}$ to $e^{3i\pi/4}$.
By an obvious modification of Proposition \ref{prop:17}\footnote{
Start with the formula $\ind=\mu(\disc{u})+2(1-\brindices{+}{\disc{u}})+k+1$.
Subtract $k+1$ because the marked points are not varying, and change the $2$ to a $1$ because $\linebundle$ has rank 1.
}, the index of the first operator is $(k+1)+1(1-(k+1))=1$.
By automatic regularity in dimension one it follows that the first operator is surjective.
The middle operator has index $k+2$, and hence the last operator has index $k+1$.
Since it is one dimensional, it is again surjective.
The snake lemma then implies that the middle operator is surjective as well.

\begin{proposition}
  \label{prop:1}
  All $\acs$-holomorphic discs with boundary on $\lagstd$ are regular.
\end{proposition}
\begin{proof}
  We proved the case when all marked points are branch jumps above.
  Any disc can be obtained from a disc with all branch jumps by adding some non-branch marked points.
  Adding non-branch marked points does not affect regularity.
  Thus all discs are regular.
\end{proof}

For strips the same result holds because we are considering $\J=\depj{t}{t}$ with $J_t=\acs$ for all $t$.
\begin{corollary}
  \label{cor:1}
  All $\J$-holomorphic strips are regular.
\end{corollary}

\section{Floer cohomology of $\lagstd$}
\label{sec:floer-cohom-lagstd}

In this section we calculate the self-Floer cohomology of the immersed Lagrangian spheres $\lagstd$ with $\zz_2$-coefficients.
We use a version of Floer cohomology that combines Morse theory and holomorphic curves.
A precise definition is given in Section \ref{sec:immers-self-lagr}.
The calculation is given in Section \ref{sec:calc-floer-cohom}.

Before getting to the details, we explain the idea behind the version of Floer cohomology that we use.
As auxiliary data we need a Morse function and time-dependent almost complex structure.
The Floer cochain complex is defined to be the Morse complex, plus two extra generators for each self-intersection point of the Lagrangian.
The Floer differential is defined by counting pearly trajectories, which are strings of holomorphic strips and Morse flow lines.
Actually because the Lagrangians are exact (and hence there are no holomorphic strips without branch jumps), there are not many types of pearly trajectories that need to be counted.
See Figure \ref{fig:trajectories} in Section \ref{sec:introduction} for the types of trajectories.

In \cite{alston-fciegl} it is proved that the Floer cohomology defined by counting pearly trajectories is well-defined and does not depend on the choice of auxiliary data.
More precisely, the following theorem is proven:
\begin{theorem}
  \label{thm:5}
  Let $\iota:L\to \bar L$ be a compact immersed exact graded Lagrangian submanifold of the exact graded symplectic manifold $(M,\omega,\lambda,\Omega)$.
  Let 
  \begin{displaymath}
    R=\set{(p,q)\in L\times L}{\iota(p)=\iota(q),\ p\neq q}.
  \end{displaymath}
  Assume $\dim L\geq 2$.
  Let $J$ be an almost complex structure on $M$ such that there exists a compact set $K\subset M$ with the property that any holomorphic disc with boundary on $\bar L$ is contained in $K$.
  Let $f:L\to\rr$ satisfy $df=\iota^*L$.
  Assume $L$ satisfies the following positivity condition: 
  \begin{displaymath}
    \text{If $(p,q)\in R$ and $f(p)-f(q)>0$ then $\ind(p,q)\geq \frac{n+3}{2}$.}
  \end{displaymath}
  Then the self-Floer cohomology of $L$ is well-defined and can be computed by counting pearly trajectories for a generic time dependent almost complex structure that agrees with $J$ outside some compact set.
\end{theorem}

See Section \ref{sec:immers-self-lagr} for the precise definition of the Floer cohomology.
Moreover, in \cite{alston-fciegl}, it is proved that this version of Floer cohomology agrees with the more standard version (at least standard for embedded Lagrangians); namely, Floer cohomology defined by taking two copies of the Lagrangian, Hamiltonian perturbing one copy, and then counting holomorphic strips connecting intersection points.

We end our preparatory remarks by saying a few words about how transversality and bubbling are dealt with in \cite{alston-fciegl}.
The strips that appear in the pearly trajectories are holomorphic with respect to a time-dependent almost complex structure $\J=\depj{t}{t}$.
The time-dependence of $\J$ allows transversality to be achieved using classical methods (essentially Section 7 of \cite{MR1360618}).
A sequence of such strips can degenerate into broken strips plus disc bubbles.
Dealing with the disc bubbles is the main difficulty.
They attach to the strips via branch points (because of exactness), and they are $J_t$-holomorphic for $t=0$ or $1$.
However, the positivity assumption in Theorem \ref{thm:5} implies that the strip component must have negative index.
Hence generically it cannot exist and disc bubbling is ruled out. 
In the terminology of \cite{fooo} and \cite{MR2785840}, the positivity assumption should be thought of as implying that the Lagrangian is unobstructed.

\subsection{Immersed self-Lagrangian Floer cohomology}
\label{sec:immers-self-lagr}
Let $\morsefun:S^2\to \rr$ be a smooth function and $\metric$ a smooth metric on $S^2$ such that the pair $\morsefun,\metric$ is Morse-Smale.
Let
\begin{displaymath}
  \morsecpx{}=\bigoplus_k\morsecpx{k}
\end{displaymath}
denote the Morse complex of $\morsefun$ with $\zz_2$-coefficients.
The degree $k$ piece is generated by the critical points of $f$ with Morse index $k$ (the negative definite space of $D^2f$ has dimension $k$).

We define the Floer chain complex of the immersion $\imm:S^2\to\lagstd$ with $\zz_2$-coefficients to be the vector space
\begin{equation}
  \floercpx{}=\morsecpx{}\oplus \zz_2\bjstd.
\end{equation}
Here $\bjstd$ is as in Definition \ref{dfn:35}.
$\floercpx{}$ has canonical generators corresponding to the critical points of $\morsefun$ and the elements of $\bjstd$.
The generators corresponding to the elements of $\bjstd$ are graded by their indices, which by Lemma \ref{lemma:10} are
\begin{eqnarray*}
  \ind (\p,\q)&=&-1,\\
  \ind (\q,\p)&=&3.
\end{eqnarray*}
$\floercpx{}$ is the graded vector space
\begin{displaymath}
  \floercpx{}=\bigoplus_{k}\floercpx{k}.
\end{displaymath}

We turn to defining the differential.
We need some preliminary definitions first.

\begin{definition}
  \label{dfn:20}
  Given the Morse-Smale pair $\morsefun,\metric$ on $S^2$ as above, let $\flow{\morsefun}{s}$ denote the time $s$ negative gradient flow of $\morsefun$, so 
  \begin{displaymath}
   \pd{}{s} \flow{\morsefun}{s} +\nabla f \circ\flow{\morsefun}{s}=0.
  \end{displaymath}
  For $x$ a critical point of $f$, let
  \begin{displaymath}
    \unst(x)=\set{y\in S^2}{\lim_{s\to-\infty}\flow{\morsefun}{s}(y)=x}.
  \end{displaymath}
  denote the \textit{unstable manifold} of $x$, and
  \begin{displaymath}
    \st(x)=\set{y\in S^2}{\lim_{s\to\infty}\flow{\morsefun}{s}(y)=x}.
  \end{displaymath}
  denote the \textit{stable manifold}.
\end{definition}

We will generally denote generators of $\floercpx{}$ using bold letters, for example $\gen{x_\pm}$ or $\gen{x}$.
We also view generators as elements of $\crit(\morsefun)\amalg\bjstd$.

\begin{definition}
  \label{dfn:10}
  For $\gen{x_\pm}$ generators of $\floercpx{}$, we define the moduli space $\connorbitstd$  as follows:
  \begin{enumerate}
  \item If $\gen{x_\pm}\in\crit(\morsefun)$ then $\connorbit(\gen{x_-},\gen{x_+})$ is the set of (unparameterized) Morse trajectories,
    \begin{displaymath}
      \connorbit(\gen{x_-},\gen{x_+})=\left(\unst(\gen{x_-})\cap \st(\gen{x_+})\right)/\rr.
    \end{displaymath}
    (If $\gen{x_-}=\gen{x_+}$, we do not mod out by $\rr$.)
  \item If $\gen{x_-}\in \bjstd$ and $\gen{x_+}\in\crit(\morsefun)$ we let
    \begin{displaymath}
      \connorbitstd= \modms{0,0}\times_{\ev_{+\infty}}\st(\gen{x_+}),
    \end{displaymath}
    with $\brjumps{-\infty}{}=\gen{x_-}$, $\brjumps{+\infty}{}=\emptyset$.
  \item If $\gen{x_-}\in\crit(\morsefun)$ and $\gen{x_+}\in \bjstd$ we let
    \begin{displaymath}
      \connorbitstd= \unst(\gen{x_-})\times_{\ev_{-\infty}}\modms{0,0}
    \end{displaymath}
    with $\brjumps{-\infty}{}=\emptyset$, $\brjumps{+\infty}{}=\gen{x_+}$.
  \item If $\gen{x_\pm}\in \bjstd$, we let $\connorbitstd$ be the union of the following two sets:
    \begin{itemize}
    \item the set $\modms{0,0}$, with $\brjumps{\pm\infty}{}=\gen{x_\pm}$; and
    \item the set of all pairs $(\disc{u_1},\disc{u_2})$ with 
      \begin{displaymath}
        \disc{u_1}\in\modms{0,0}
      \end{displaymath}
      where $\brjumps{-\infty}{}=\gen{x_-}$ and $\brjumps{+\infty}{}=\emptyset$, and
      \begin{displaymath}
        \disc{u_2}\in\modms{0,0}
      \end{displaymath}
      where $\brjumps{-\infty}{}=\emptyset$ and $\brjumps{+\infty}{}=\gen{x_+}$, and such that
      \begin{displaymath}
        \flow{\morsefun}{s}(\ev_{+\infty}(\disc{u_1}))=ev_{-\infty}(\disc{u_2})
      \end{displaymath}
      for some $s>0$.
    \end{itemize}
  \end{enumerate}
\end{definition}

\begin{remark}
  \label{rmk:1}
  The positivity assumption in Theorem \ref{thm:5} (which the immersed spheres $\lagstd$ satisfy) rules out the existence of the latter types of trajectories in item 4.\ when $\deg(\gen{x_-})-\deg(\gen{x_+})=1$.
\end{remark}

\begin{lemma}(See \cite{alston-fciegl})
  \label{lemma:30}
  For generic $(\morsefun,\metric,\J=\depj{t}{t})$, $\connorbitstd$ is a smooth manifold of dimension $\deg(\gen{x_-})-\deg(\gen{x_+})-1$.
\end{lemma}

\begin{definition}
  \label{dfn:13}
  The \textit{Floer differential}
  \begin{displaymath}
    \floerdiff:\floercpx{}\to\floercpx{}
  \end{displaymath}
  is defined on generators $\gen{x_+}$ by the formula
  \begin{displaymath}
    \floerdiff(\gen{x_+})=\sum_{\gen{x_-}}\# \connorbitstd \cdot \gen{x_-},
  \end{displaymath}
  where the sum is over all $\gen{x_-}$ such that $\dim \connorbitstd = 0$.
\end{definition}

\begin{definition}
  \label{dfn:12}
  The \textit{Floer cohomology} $\floercoh{}$ of $\imm$ with $\zz_2$-coefficients is the cohomology of the complex $(\floercpx{},\floerdiff)$.
\end{definition}

By Theorem \ref{thm:5}, $\floercoh{}$ is well-defined and independent of the generic choice of data used to define it.
The following definition specifies what it means for data to be generic.
\begin{definition}
  \label{dfn:4}
  The data $\morsefun,\metric,\J=\depj{t}{t}$ is generic if the following conditions are satisfied:
  \begin{enumerate}
  \item $\morsefun,\metric$ is Morse-Smale.
  \item $\p,\q\in S^2$ are contained in the top-dimensional stable and unstable manifolds of $(\morsefun,\metric)$.
  \item For all $k_0,k_1,\metabrjumps$, $\modms{k_0,k_1}$ is regular at non-constant strips.
  \item For all $k_0,k_1,\metabrjumps$, $\ev_{\pm\infty}:\modms{k_0,k_1}\to S^2\amalg \bjstd$ is transverse at non-constant strips to all stable and unstable manifolds of $(\morsefun,\metric)$.
  \item For all $k_0,k_1,\metabrjumps,k_0',k_1',\metabrjumps'$, $\ev_{+\infty}:\modms{k_0,k_1}\to S^2\amalg\bjstd$ is transverse to $\ev_{-\infty}:\modmscust{k_0',k_1'}{\str;\metabrjumps';\J}\to S^2\amalg\bjstd$ at non-constant strips.
  \end{enumerate}
\end{definition}

By the results of Section \ref{sec:class-holom-curv}, taking $\J=\depj{t}{t}$ with $J_t=\acs$ for all $t$ and $\morsefun,\metric$ generic will satisfy the conditions of the definition.

\subsection{Calculation of Floer cohomology}
\label{sec:calc-floer-cohom}

Let
\begin{displaymath}
  f:S^2\to \rr
\end{displaymath}
be a function with two critical points, one at $(s,e^{i\theta})=(0,1)=:\pmax$ which is a global max and one at $(s,e^{i\theta})=(0,-1):=\pmin$ which is a global min.
Here, $(s,e^{i\theta})$ are cylindrical coordinates from equation \eqref{eq:5}.
Let $\metric$ be the standard metric induced by the embedding $S^2\subset\rr^3$ from \eqref{eq:5}; we may assume $(f,\metric)$ is Morse-Smale.

The Floer cochain complex with this data becomes
\begin{displaymath}
  \floercpx{}=\floercpx{-1}\oplus \floercpx{0} \oplus \floercpx{2} \oplus \floercpx{3}
\end{displaymath}
with
\begin{eqnarray*}
  \label{eq:6}
  \floercpx{-1} &=& \zz_2\cdot (\p,\q),\\
  \floercpx{0} &=& \zz_2\cdot \pmin,\\
  \floercpx{2} &=& \zz_2\cdot \pmax,\\
  \floercpx{3} &=& \zz_2 \cdot (\q,\p).
\end{eqnarray*}

For degree reasons, the only nontrivial things to calculate are $\floerdiff((\p,\q))$ and $\delta(\pmax)$.
By Definitions \ref{dfn:10} and \ref{dfn:12}, $\floerdiff((\p,\q))$ counts elements of
\begin{displaymath}
  \connorbit(\pmin,(\p,\q))= \unst(\pmin)\times_{\ev_{-\infty}}\modms{0,0}.
\end{displaymath}
Since $\unst(\pmin)=\sett{\pmin}$, the $-\infty$ end of the strip must pass through $\pmin$.
The $+\infty$ end of the strip must have a branch jump of type $(\p,\q)$.
Propositions \ref{prop:4} and \ref{prop:5} can be used to find all such strips.

Similarly, $\floerdiff(\pmax)$ is defined by counting elements of
\begin{displaymath}
  \connorbit((\q,\p),\pmax)= \modms{0,0}\times_{ev_{+\infty}}\st(\pmax).
\end{displaymath}
This space consists of holomorphic strips such that the $-\infty$ end has a branch jump of type $(\q,\p)$ and the $+\infty$ end passes through $\pmax$.

\begin{lemma}
  \label{lemma:1}
  For the immersion $\imm:S^2\to\lagstd$ with $r\in\sett{1,\ldots,N}$,
  \begin{displaymath}
   \#\connorbit(\pmin,(\p,\q))=\#\connorbit((\p,\q),\pmax)=2^{r-1}.
  \end{displaymath}
\end{lemma}
\begin{proof}
  Let $\strip{u}$ be a strip contributing to $\connorbit(\pmin,(\p,\q))$.
  Think of the domain of $u$ as $\dd\setminus\sett{z_0,z_1}$, with $z_0=1$ corresponding to $+\infty$ and $z_1$ corresponding to $-\infty$.
  Write $u=(f,g,rh)$ with respect to this identification.
  The only branch point is at $1$, so $h^{-1}(1)=1$.
  Proposition \ref{prop:5} then implies that up to reparameterization $h(z)=z$ and the moduli space $\modms{0,0}$ has $2^{r-1}$ components, each diffeomorphic $S^1\times\rr$.
  The $S^1$ corresponds to changing $e^{i\theta}$ in the proposition, and the $\rr$ corresponds to where the marked point $z_1$ is located on the boundary arc $\bdy\dd\setminus\sett{1}$. 
  In fact, the map
  \begin{displaymath}
    \ev_{-\infty}:\modms{0,0}\to S^2
  \end{displaymath}
  restricts to a diffeomorphism from each component onto the cylindrical coordinate patch in $S^2$ defined in \eqref{eq:1}.
  Thus exactly one element of each component of $\modms{0,0}$ satisfies $\ev_{-\infty}(\strip{u})=\pmin$, and hence
  \begin{displaymath}
    \#\connorbit(\pmin,(\p,\q))=2^{r-1}.
  \end{displaymath}
  
  The proof that $\#\connorbit((\p,\q),\pmax)=2^{r-1}$ is similar.
\end{proof}

An immediate corollary is the main theorem of this paper.
\begin{theorem}
  \label{thm:1}
  The self-Floer cohomology with $\zz_2$-coefficients of the immersion
  \begin{displaymath}
    \imm:\lagstd\to\mnfld,\quad r\in\sett{1,\ldots,N}
  \end{displaymath}
  is
  \begin{displaymath}
    \floercoh{}=\left\{
      \begin{array}{ll}
        0 & r=1,\\
        (\zz_2)^4 & r>1.
      \end{array}
    \right.
  \end{displaymath}
  More precisely, in case $r>1$, the cohomology is one-dimensional in degrees $-1,0,2,3$; and $0$ elsewhere.
\end{theorem}

\section{Generalization to other immersed spheres}
In Section \ref{sec:geom-lefsch-fibr} we use the geometry of Lefschetz fibrations to construct a wider class of immersed Lagrangian spheres in $\mnfld$.
In Section \ref{sec:floer-cohom-immers} we explain how the calculations used to prove Theorem \ref{thm:1} can be generalized to calculate the self-Floer cohomology for spheres in this larger class.
\subsection{Geometry of the Lefschetz fibration}
\label{sec:geom-lefsch-fibr}
Recall that we have a Lefschetz fibration
\begin{displaymath}
  \leffib:\mnfld\to\cc,\ (\za,\zb,\zc)\mapsto\zc.
\end{displaymath}
Let $\text{Critv}(\pi)=\sett{1,\ldots,N}$ be the critical values of $\pi$ and let $\text{Crit}(\pi)\subset \mnfld$ be the critical points.
Away from $\text{Crit}(\pi)$, $\mnfld$ has a canonical vertical tangent bundle $T^v\mnfld$, and a canonical horizontal tangent bundle $T^h\mnfld$ defined by
\begin{displaymath}
  T_p^h\mnfld=\set{V\in T_p\mnfld}{\iota_V\sympl|T_p^v\mnfld=0}.
\end{displaymath}
The horizontal tangent space defines parallel translation maps
\begin{displaymath}
  PT_\gamma: \mnfldcust{\gamma(a)}\to \mnfldcust{\gamma(b)},
\end{displaymath}
where $\gamma:[a,b]\to\cc$ is a piece-wise smooth map.

Let $\gamma:[a,b]\to \cc$ be a smooth embedded path with $\gamma^{-1}(\text{Critv}(\pi))=\sett{b}$, and let $q\in\text{Crit}(\pi)$ be the unique critical point in $\mnfldcust{\gamma(b)}$; $\gamma$ is called a \textit{vanishing path}.
Let
\begin{displaymath}
  V_{\gamma}=\set{p\in \mnfldcust{\gamma(a)}}{\lim_{t\to b}{PT}_{\gamma|[a,t]}(p)=q}.
\end{displaymath}
$V_{\gamma}$ is the \textit{vanishing cycle} associated to the path $\gamma$, and $V_{\gamma}$ is a Lagrangian sphere in the symplectic submanifold $\mnfldcust{\gamma(a)}$.
Let 
\begin{displaymath}
  \Delta_\gamma=\bigcup_{a\leq t<b} V_{\gamma|[t,b]}\bigcup \sett{q}.
\end{displaymath}
$\Delta_\gamma$ is the \textit{Lefschetz thimble} associated to $\gamma$ and it is a Lagrangian disk in $\mnfld$.
The following is a standard fact, see for example \cite{seidel-fcpclt}.
\begin{lemma}
  \label{lemma:11}
  Let $L\subset \mnfld$ be a submanifold, and assume that $\pi|L:L\to \cc$ is a fibration over some embedded curve $C\subset\cc$.
  Assume furthermore that $L_c=L\cap \mnfldcust{c}$ is a Lagrangian submanifold of $\mnfldcust{c}$ for every $c\in C$.
  Then $L$ is Lagrangian if and only if the parallel transport maps over $C$ map the $L_c$'s into the $L_c$'s.
\end{lemma}
Note that the hypothesis that $L_c$ is Lagrangian is always true because the real dimension of $\mnfldcust{c}$ is 2.

\begin{lemma}
  \label{lemma:18}
  $\hamone=\frac{1}{2}(|\za|^2-|\zb|^2)$ is invariant under parallel transport; in other words, if $X$ is a horizontal vector field then $X(\hamone)=0$.
\end{lemma}
\begin{proof}
  $\hamvf{\hamone}(\za,\zb,\zc)=(i\za,-i\zb,0)$ is a vertical vector field.
  Hence
  \begin{displaymath}
    X(\hamone)=d\hamone(X)=\omega(\hamvf{\hamone},X)=0
  \end{displaymath}
  for any horizontal $X$.
\end{proof}

\begin{lemma}
  Let $\gamma:[a,b]\to\cc$ be a vanishing path.
  Then
  \begin{displaymath}
    \Delta_\gamma=\set{(\za,\zb,\gamma(t))\in \mnfld}{t\in[a,b],\ |\za|=|\zb|}.
  \end{displaymath}
\end{lemma}
\begin{proof}
  $\hamone=0$ on $\text{Crit}(\pi)$, so $\hamone|\Delta_\gamma=0$ by Lemma \ref{lemma:18}.
  The intersection of the hypersurface $\hamone=0$ with a fiber $\mnfldcust{\gamma(t)}$ is precisely the set of points in $\mnfld$ of the form
  \begin{displaymath}
    (\za,\zb,\gamma(t))
  \end{displaymath}
  with $|\za|=|\zb|$.
  This is a circle, and hence must be the vanishing cycle in the fiber $\mnfldcust{\gamma(t)}$.
  Since $\Delta_{\gamma}$ is the union of all vanishing cycles over all points in the image of $\gamma$, the result follows.
\end{proof}

Now let $\gamma:[a,b]\to\cc$ be a smooth path with $\gamma^{-1}(\text{Critv}(\pi))=\sett{a,b}$.
Assume $\gamma$ is embedded, except possibly with $\gamma(a)=\gamma(b)$, in which case we require $-\dot\gamma(a)$ and $\dot\gamma(b)$ to not be positive multiples of each other.
The previous lemma implies that $\gamma$ is a \textit{matching path}, meaning that the vanishing cycles coming from the critical value $\gamma(a)$ are the same as the vanishing cycles coming from the critical value $\gamma(b)$.
The Lagrangian submanifold
\begin{displaymath}
    \Sigma_\gamma=\set{(\za,\zb,\gamma(t))\in \mnfld}{t\in[a,b],\ |\za|=|\zb|}
\end{displaymath}
is the \textit{matching cycle} associated to the matching path $\gamma$.
Clearly, if $\gamma(a)\neq \gamma(b)$ then $\Sigma_\gamma$ is an embedded Lagrangian sphere.
If $\gamma(a)=\gamma(b)$ then we get an immersed Lagrangian sphere:
\begin{lemma}
  \label{lemma:17}
  Let $\gamma$ be a matching path with $\gamma(a)=\gamma(b)$ and such that $-\dot\gamma(a)$ and $\dot\gamma(b)$ are not positive real multiples of each other.
  Then the matching cycle $\Sigma_\gamma$ is an immersed Lagrangian sphere with one transverse self-intersection point.
\end{lemma}
The fact that the image intersects itself transversely at the self-intersection point follows from the next lemma.
\begin{lemma}
  \label{lemma:19}
  Let $\gamma:[0,1]\to\cc$ be a vanishing path.
  Let $j=\gamma(1)$ be the critical value and let $q\in \mnfld$ be the critical point with $\pi(q)=j$.
  Let 
  \begin{displaymath}
    \xi=\sqrt{\dot\gamma(1)}\sqrt{j-1}\cdots\sqrt{j-(j-1)}\sqrt{j-(j+1)}\cdots\sqrt{j-N}\in\cc.
  \end{displaymath}
  Then
  \begin{displaymath}
    T_q\Delta_{\gamma}=\Span_\rr\sett{(\xi,\xi,0),\ (i\xi,-i\xi,0)}.
  \end{displaymath}
\end{lemma}
\begin{proof}
  Parameterize $\Delta_\gamma\setminus\sett{q}$ as
  \begin{gather*}
    (0,1)\times S^1\to \mnfld,\\
    (t,\theta)\mapsto (e^{i\theta}\sqrt{\gamma-1}\cdots\sqrt{\gamma-N},e^{-i\theta}\sqrt{\gamma-1}\cdots\sqrt{\gamma-N},\gamma).
  \end{gather*}
  Let $f=\sqrt{\gamma-1}\cdots\sqrt{\gamma-N}/\sqrt{\gamma-j}$, so $f$ is smooth.
  Then
  \begin{displaymath}
    (f\sqrt{\gamma-j})^{-1}\pd{}{\theta}  = (ie^{i\theta},-ie^{-i\theta},0).
  \end{displaymath}
  Also
  \begin{displaymath}
    \lim_{t\to1}\frac{f\sqrt{\gamma-1}}{|f\sqrt{\gamma-1}|}=\frac{\xi}{|\xi|}.
  \end{displaymath}
  Thus
  \begin{displaymath}
    \lim_{t\to1} |f\sqrt{\gamma-1}|^{-1}\pd{}{\theta}=\lim_{t\to 1} \frac{f\sqrt{\gamma-1}}{|f\sqrt{\gamma-1}|}(f\sqrt{\gamma-1})^{-1}\pd{}{\theta}=\frac{\xi}{|\xi|}(ie^{i\theta},-ie^{-i\theta},0).
  \end{displaymath}
  Plugging in $\theta=0,\pi/2$ gives the result.
\end{proof}

The immersed spheres $\lagstd$ are of course the matching cycles for the paths
\begin{displaymath}
  \gamma_{N,r}(t)=re^{2\pi i t},\quad 0\leq t\leq 1.
\end{displaymath}

We record here a useful lemma about horizontal vectors.
The proof is a straightforward calculation.
\begin{lemma}
  \label{lemma:25}
  Let $\zc\in\cc$, $\lambda\in T_{\zc}\cc=\cc$, and $p=(\za,\zb,\zc)\in\mnfld$.
  Then the horizontal lift of $\lambda$ to $T_p^h\mnfld$ is
  \begin{displaymath}
    \tilde\lambda=\left(\frac{c\lambda\bar \zb}{|\za|^2+|\zb|^2},\frac{c\lambda\bar\za}{|\za|^2+|\zb|^2},\lambda\right),
  \end{displaymath}
  where
  \begin{displaymath}
    c=\sum_{j=1}^N\prod_{k\neq j}(\zc-k).
  \end{displaymath}
\end{lemma}

\subsection{Floer cohomology of immersed spheres}
\label{sec:floer-cohom-immers}
Let $L=\Sigma_\gamma$ be an immersed Lagrangian sphere as in Lemma \ref{lemma:17}, and $\iota:S^2\to L$ an immersion.
Parameterize $\gamma$ so that the domain of $\gamma$ is $[0,1]$ and $\gamma(t)$ moves around the image of $\gamma$ in the counterclockwise direction as $t$ increases.
Let $\q\in S^2$ be such that $\leffib(\iota(\q))=\gamma(0)$ and $\p\in S^2$ such that $\leffib(\iota(\p))=\gamma(1)$.
$\iota$ is an exact immersion because the domain $S^2$ satisfies $H^1(S^2)=0$, and $\iota$ can be graded because $\pi_1(S^2)=0$.

\begin{lemma}
  \label{lemma:5}
  Let $p\in L$ with $\leffib(p)=\gamma(t)$.
  Then
  \begin{displaymath}
    \Det^2_{\holvf}(T_pL)=\frac{i\dot\gamma(t)^2}{|\dot\gamma(t)^2|}.
  \end{displaymath}
\end{lemma}
\begin{proof}
  The proof is very similar to the proof of Lemma \ref{lemma:4}, except that $\hamvf{\hamtwo}$ needs to be replaced with the horizontal vector field $\tilde \lambda$ that is a lift of $\dot\gamma(t)$.
  Lemma \ref{lemma:25} gives a formula for $\tilde\lambda$.
  Then, picking up near the end of the proof of Lemma \ref{lemma:4}, we get
  \begin{displaymath}
    \holvf(\hamvf{\hamone},\tilde\lambda)=i\dot\gamma.
  \end{displaymath}
  Thus
  \begin{displaymath}
    \Det^2_{\holvf}(T_PL)=\frac{i\dot\gamma(t)^2}{|\dot\gamma(t)^2|}.
  \end{displaymath}
\end{proof}

\begin{corollary}
  \label{cor:3}
  With respect to any grading for $L$, we have
  \begin{displaymath}
    \ind(\p,\q)=-1,\quad \ind(\q,\p)=3.
  \end{displaymath}
\end{corollary}
\begin{proof}
  The index is invariant under deformations of $\gamma$ that do not pass through critical points, and are such that $\dot\gamma(0)$ is always different from $-\dot\gamma(1)$.
  Thus we can assume that the image of $\gamma$ is a circle and $\dot\gamma(0)=\dot\gamma(1)$.
  The corollary is then a simple calculation using Lemma \ref{lemma:5} and Definition \ref{dfn:35}.
\end{proof}

Now pick a Morse function $f:S^2\to \rr$ such that $f$ has two critical points; call them $\pmin$ and $\pmax$.
Assume also that $\leffib\circ\iota(\pmin)\neq\leffib\circ\iota(\pmax)$, and also that neither $\pmin$ nor $\pmax$ is a singular point of $L$.
With this data, the Floer cochain complex $\cf(\iota)$ has four generators, one each in degrees $-1,0,2,3$.
To calculate the differential we need to know the holomorphic strips bounded by $L$.

Examination of the classification of discs on $\lagstd$ carried out in Section \ref{sec:classifying-strips} shows that similar results hold for $L$.
The main differences are the following.
First, let $\phi$ be a biholomorphism from $D^2$ to the closure of the interior of the image of $\gamma$ such that $\phi(1)=\gamma(0)=\gamma(1)$.
Then if $u=(f,g,h)$ is a strip with boundary on $\lagstd$, then $\phi\circ h \circ \phi^{-1}:D^2\to D^2$ is a Blaschke product.
Second, the key thing that determines the number of discs bounded by $L$ is the number of critical values of $\leffib$ contained in the interior of the image of $\gamma$.
This can be seen in the proof of Proposition \ref{prop:4}: In the equation $fg=(rh-1)\cdots(rh-N)$, the term $rh-j$ can be factored (factors of which contribute to $f$ and $g$) in multiple ways if and only $j$ is in the interior of $\gamma$.\footnote{In the notation of Proposition \ref{prop:4}, if $j\geq r$ then $rh-j$ contributes $\sqrt{rh-j}$ to both $f$ and $g$. If $j<r$, then $rh-j$ can factor as
  \begin{displaymath}
    rh-j=(\xi_j\sqrt{jh-r})(\eta_j\sqrt{jh-r}),
  \end{displaymath}
with the first factor contributing to $f$ and the second factor contributing to $g$.
}
We thus get the following generalization of Lemma \ref{lemma:1}.
\begin{lemma}
  \label{lemma:6}
  Let $C$ be the number of critical values of $\leffib:\mnfld\to\cc$ contained in the interior of the image of $\gamma$.
  Then
  \begin{displaymath}
   \#\connorbit(\pmin,(\p,\q))=\#\connorbit((\p,\q),\pmax)=2^{C}.
  \end{displaymath}
\end{lemma}

An immediate corollary is a generalization of Theorem \ref{thm:1}.
\begin{theorem}
  \label{thm:2}
  Let $\iota:S^2\to L=\Delta_\gamma$ be the immersed sphere over the matching path $\gamma$ as above.
  Let $C$ be the number of critical values of $\leffib$ contained in the interior of the image of $\gamma$.

  If $C=0$ then
  \begin{displaymath}
    \floercohcust{}{\iota}\cong0;
  \end{displaymath}
  if $C\geq 1$ then
  \begin{displaymath}
    \floercohcust{}{\iota}\cong (\zz_2)^4.
  \end{displaymath}
  More precisely, in the latter case, the Floer cohomology has dimension $1$ in degrees $-1,0,2,3$.
\end{theorem}

\bibliographystyle{amsplain}
\bibliography{mybib}

\end{document}